\documentclass[11pt, reqno]{amsart}

\usepackage{amsmath,amssymb,amsthm, epsfig}
\usepackage{amsmath}
\usepackage{amssymb}
\usepackage{color}
\usepackage{mathtools}
\usepackage{calrsfs}
\DeclareMathAlphabet{\pazocal}{OMS}{zplm}{m}{n}
\usepackage{epsfig}
\usepackage[mathscr]{eucal}
\usepackage[latin1]{inputenc}
\usepackage{enumerate}

\usepackage{xcolor}
\usepackage{hyperref}
\usepackage[nameinlink]{cleveref}
\hypersetup{
	colorlinks,
	linkcolor={blue!50!black},
	citecolor={blue!50!black},
	urlcolor={blue!50!black}
}

\newtheorem{theorem}{Theorem}
\newtheorem{definition}{Definition}

\newtheorem{lemma}{Lemma}
\newtheorem{proposition}{Proposition}
\newtheorem{corollary}{Corollary}

\date{}
\numberwithin{equation}{section}
\numberwithin{theorem}{section}
\numberwithin{lemma}{section}
\numberwithin{corollary}{section}
\numberwithin{remark}{section} \numberwithin{proposition}{section}
\numberwithin{definition}{section}

\usepackage[margin=1in]{geometry}

\newcommand{\dd}{\mathrm{d}}

\begin{document}

\title[Regularity of solutions of American option pricing]{An obstacle problem arising from American options pricing: regularity of solutions}

\author[H. Borrin]{Henrique Borrin}
\address{Departamento de Matem\'{a}tica Pura e Aplicada, Universidade Federal do Rio Grande do Sul, Porto Alegre - RS, Brazil.}
\email{henrique.borrin@ufrgs.br}

\author[D. Marcon]{Diego Marcon}
\address{Departamento de Matem\'{a}tica Pura e Aplicada, Universidade Federal do Rio Grande do Sul, Porto Alegre - RS, Brazil.}
\email{diego.marcon@ufrgs.br}

\begin{abstract}
	We analyse the obstacle problem for the nonlocal parabolic operator
	\[
	\partial_t u + (-\Delta)^{s} u - b \cdot \nabla u  - \pazocal{I}u - ru,
	\] where $b\in\mathbb{R}^n$, $r\in\mathbb{R}$, and $\pazocal{I}$ is a nonlocal lower order diffusion operator with respect to the fractional Laplace operator $(-\Delta)^{s}$. This model appears in the study of American options pricing when the stochastic process governing the stock price is assumed to be a purely jump process. We study the existence and the uniqueness of solutions to the obstacle problem, and we prove optimal regularity of solutions in space, and almost optimal regularity in time.  
	
%%%%%	This problem is motivated by an American option model proposed by Menton which introduces, into the theory of option evaluation, discontinuous paths in the dynamics of the stock's prices.
%%%%%	
%%%%% extend results of Caffarelli-Figalli \cite{CaFi-08} 

\bigskip

\noindent \textbf{Keywords:} Obstacle problem, fractional Laplacian, nonlocal operators, optimal regularity, free boundary problems.

\bigskip

\noindent \textbf{AMS Subject Classifications (2020):} 35K55, 35R11, 35R35, 35S10.
\end{abstract}

\maketitle

\section{Introduction}
Problems with nonlocal and fractional diffusion have been extensively studied in recent years. One of many motivations is the model of a discontinuous variation of stock prices, incorporated by Merton \cite{merton}, into the celebrated Black-Scholes equation \cite{finance}. Recall that American options allow the holder to exercise option rights before the maturity date, in constrast to European options; see \cite{finance}. Let $u$ be the rational price of an American option with payoff $\psi$. We may split the domain of $u$ into $\{u>\psi\}$ and $\{u=\psi\}$, which are known as continuation and exercise regions, respectively. These names suggest that the first time we enter the exercise region, it is optimal to exercise the option; otherwise, that is, when we are in the continuation region, we should continue the evolution of $u$. This information is encapsulated in the following obstacle problem:

\begin{equation}\label{american}
\begin{cases}
\min\left\{\partial_t u -\tfrac{1}{2}\sigma \colon D^2 u-b\cdot \nabla u-ru-\pazocal{K}u,u-\psi\right\}=0;\\
u(0,\cdot)=\psi,
\end{cases}
\end{equation}
where $r \in \mathbb{R}$ is known as the short rate, $b = (d_1+\tfrac{1}{2}\sigma_{11}-r,\dots, d_n+\tfrac{1}{2}\sigma_{nn}-r) \in \mathbb{R}^{n}$ with $d$ being the continuously compounded dividend rate of the stock, $\sigma$ is the volatility matrix of the stock, and
\[
\pazocal{K}v(t,x)\coloneqq\int_{\mathbb{R}^n}\Big[v(t,x+y)-v(t,x)-\sum_{i=1}^n(e^{y_i}-1)x_{i}\partial_{x_i}v(t,x)\Big]\, \dd \mu (y)
\]
with $\mu$ being the associated jump measure.

If there is no jump term, that is, if $\mu\equiv 0$, the regularity of \eqref{american} is well-known; see, for instance, \cite{laurencesalsa}. In this paper, we assume $\sigma\equiv 0$ so that all the regularity comes from the jump term. As in \cite{CaFi-08}, we further assume that $\dd \mu (y)$ behaves as $|y|^{-n-2s}\dd y$ as leading order, so that
\[
\pazocal{K}v(t,x)\approx -(-\Delta)^s v(t,x)+\pazocal{I}v(t,x),
\]
where $(-\Delta)^s$ is the fractional Laplacian (with respect to $x$) defined by
\[
-(-\Delta)^s f(x)\coloneqq c_{n,s} \int_{\mathbb{R}^n}\frac{f(y)-f(x)}{|y-x|^{n+2s}}\,\dd y,
\]
and $\pazocal{I}$ is a non-local operator of lower order diffusion with respect to $(-\Delta)^s$. Under this assumption, the regularity of the solution of \eqref{american} depends on the parameter $s$:
\begin{itemize}
\item  if $s<1/2$, the gradient term is of greater order with respect to $(-\Delta)^s$, and we do not expect any regularity result for $u$;
\item the case $s=1/2$ is critical in the sense that both the gradient and the fractional Laplacian have the same order, and the problem becomes very delicate;
\item  if $s>1/2$, we expect the diffusion provided by the fractional Laplacian to dominate and $u$ should be as regular as the solution of the obstacle problem for the fractional heat operator.
\end{itemize}

Our main goal is to show that the expected regularity in fact holds when $s>1/2$. More precisely, we prove optimal regularity in space and almost optimal regularity in time for continuous viscosity solutions of 
\begin{equation}\label{prin}
\begin{cases}
\min \{ \partial_t u + (-\Delta)^{s} u - b \cdot \nabla u - \pazocal{I}u - ru, \ u - \psi \} = 0 & \text{ in } (0,T] \times \mathbb{R}^{n}, \\
u(0,x)  = \psi(x) & \text{ in } \mathbb{R}^{n},
\end{cases}
\end{equation} where
\begin{enumerate}[$(i)$]
\item\label{psi} the obstacle $\psi:\mathbb{R}^{n}\longrightarrow \mathbb{R}^+$ is assumed to be a function of class $W^{2,\infty}(\mathbb{R}^n)\cap C^{2}(\mathbb{R}^n)$;
\item\label{defb} $b\in \mathbb{R}^n$ is a constant vector;
\item\label{defr} $r\in\mathbb{R}$ is a constant\footnote{Although the condition $r \ge 0$ might seem natural, see \cite{CaFi-08}, our main result holds even for $r< 0$.};%\footnote{We remark that although it is suggested in \cite[Section 5]{CaFi-08} to assume $r\geq 0$, the results of this paper are still valid if $r< 0$.}
\item\label{K} $\pazocal{I}$ is a non-local, {convex}\footnote{The convexity of $\pazocal{I}$ is only used in \Cref{lem:semiconvexity}. Thus, if $u$ is a semiconvex solution of \eqref{prin} for a possibly nonconvex operator $\pazocal{I}$, the results of this paper still hold.}, translation-invariant, uniformly elliptic operator with respect to $\mathcal{L}_0$. The latter means that for all $v, \,w\in C^{2\sigma+0^+}(x)$ which satisfy\footnote{We recall that $\phi$ is said to be $C^{2\sigma+0^+}$ punctually at $x$ if there exists $v\in\mathbb{R}^n$ and $M>0$ such that $|\phi(x+y)-\phi(x)|\leq M|y|^{2\sigma+0^+}$ if $2\sigma+0^+\leq 1$ and $|\phi(x+y)-\phi(x)-v\cdot y|\leq M|y|^{2\sigma-1+0^+}$ if $2\sigma+0^+>1$ for small $y$.}
\[\int_{\mathbb{R}^n}\frac{|v(y)|+|w(y)|}{1+|y|^{n+2\sigma}}\,\dd y<\infty,\]
we have that $\pazocal{I}v(x)$ and $\pazocal{I}w(x)$ are well defined and
\begin{equation}\label{propertyI}
M^-_{\mathcal{L}_0}(v-w)(x)\leq \pazocal{I}v(x)-\pazocal{I}w(x)\leq M^+_{\mathcal{L}_0}(v-w)(x),
\end{equation}
where $\mathcal{L}_0$ is the set of operators $L$ such that
\begin{equation}\label{definitionL}
L u(x)\coloneqq \int_{\mathbb{R}^n}\delta u(x,y)K(y)\dd y, \quad \frac{\lambda}{|y|^{n+2\sigma}}\leq K(y)\leq  \frac{\Lambda}{|y|^{n+2\sigma}}, \text{ and} \quad K(y)=K(-y).
\end{equation}
The extremal operators $M^+_{\mathcal{L}_0}$ and $M^-_{\mathcal{L}_0}$ are analogous to Pucci operators
\[\begin{split}
M^+_{\mathcal{L}_0}u(x)\coloneqq \sup_{L\in \mathcal{L}_0} Lu(x)\equiv \int_{\mathbb{R}^n} \frac{\Lambda (\delta u(x,y))^+-\lambda (\delta u(x,y))^-}{|y|^{n+2\sigma}}\dd y,\\
M^-_{\mathcal{L}_0}u(x)\coloneqq \inf_{L\in \mathcal{L}_0} Lu(x)\equiv \int_{\mathbb{R}^n} \frac{\lambda (\delta u(x,y))^+-\Lambda (\delta u(x,y))^-}{|y|^{n+2\sigma}}\dd y,
\end{split}\]
where $\delta u(x,y)\coloneqq u(x+y)+u(x-y)-2u(x)$. For simplicity, we assume $\pazocal{I}(0)=0$. We assume further that $\pazocal{I}$ is a lower order diffusion operator when compared to the fractional Laplacian $(-\Delta)^s$ in the sense that $s>\sigma>0$.
Notice that since $L u \in C^{\alpha}(\mathbb{R}^n)$ whenever $u\in C^{2\sigma+\alpha}(\mathbb{R}^n)$ for some $\alpha>0$, we have $\pazocal{I}u \in C^{\alpha}(\mathbb{R}^n)$.
\end{enumerate}
Quintessential examples of $\pazocal{I}$ include linear operators as $-(-\Delta)^\sigma$ and $L$, as well as more sophisticated nonlinear examples studied in \cite{ellipticity}, such as 
\[
\pazocal{I}u= \sup_{\beta}L_\beta u, \quad \pazocal{I}u(x)= \int_{\mathbb{R}^n}\frac{G(u(x+y)-u(x))}{|y|^{n+2\sigma}}\dd x,
\] 
where $K_{\beta}$ satisfy \eqref{definitionL} uniformly with respect to $\beta$ and $G$ is a convex monotone Lipschitz function with $G(0)=0$.

In light of \cite[Section 2]{silvestre}, we remark that since we need the well posedness of the inverse of the fractional Laplacian and we assume that $s>1/2$, we consider throughout the paper that the dimension satisfies $n\geq 2$.
The main result of this paper is the following:

\begin{theorem}\label{maintheorem}
	Assume that $\psi$, $b$, $r$, and $\pazocal{I}$ satisfy conditions \eqref{psi}, \eqref{defb}, \eqref{defr}, and \eqref{K}, respectively. Then, there exists a unique (continuous) viscosity solution $u$ of \eqref{prin}. Moreover, $u$ is globally Lipschitz on $(0,T]\times\mathbb{R}^n$, and we have 
	\[
		\partial_t u\in C_{t,x}^{\frac{1-s}{2s}-0^+,1-s}((0,T]\times\mathbb{R}^n) \qquad \text{ and } \qquad
		(-\Delta)^su \in C_{t,x}^{\frac{1-s}{2s},1-s}((0,T]\times\mathbb{R}^n).
     \]
\end{theorem}

We were unable to find the existence of solutions for the obstacle problem \eqref{prin} in the literature and we present a proof in \Cref{sect2}. Nevertheless, in the elliptic case, Petrosyan and Pop \cite{ellipticcase} prove existence and regularity results when $\pazocal{I}\equiv 0$, $b \in C^{s}(\mathbb{R}^{n};\mathbb{R}^{n})$, $r \in C^{s}(\mathbb{R}^{n})$ is negative bounded away from zero, and $\psi\in C^{3s}$.

Our regularity in time would be optimal if we did not have $0^{+}$ in the Hölder exponent of $\partial_{t} u$. It is natural to expect that
\[
\partial_t u\in C_{t,x}^{\frac{1-s}{2s},1-s}((0,T]\times\mathbb{R}^n);
\] however, this is unknown even for the fractional heat operator. In fact, the same type of regularity  of solutions has been addressed by Caffarelli and Figalli \cite{CaFi-08} when $b \equiv 0$, $\pazocal{I} \equiv 0$, and $r \equiv 0$. In the setting of \cite{CaFi-08}, the regularity of the free boundary for the obstacle problem has been investigated by Barrios, Figalli, and Ros-Oton \cite{fb}, where they show the free boundary is of class $C^{1,\alpha}$ in space-time. Their techniques, however, heavily depend on the scale invariance of the operator, do not readily extend to the general problem \eqref{prin}, and can be the subject of a future work.

\subsection*{Aknowledgements} Henrique is partially supported by CAPES through a Master's scholarship. Diego is partially supported by CNPq-Brazil through grant 311354/2019-0.

\section{Comparison results and first regularity estimates}\label{sect2}
We first recall the general definition of a viscosity solution for a nonlocal problem $\pazocal{L}u=f$, where $f$ is a bounded continuous function and $\pazocal{L}u= \partial_t u + (-\Delta)^{s} u - b \cdot \nabla u - \pazocal{I}u - ru$:
\begin{definition}\label{auxdefviscosity} An upper semicontinuous function $u$ on $(0,T]\times \mathbb{R}^n$ is a subsolution of $\pazocal{L}u=f$ at $(t_0,x_0)$  if for all functions $\phi\in C^{1,2}(\overline{B_R(t_0,x_0)})$ such that $0=(u-\phi)(t_0,x_0)> (u-\phi)(t,x)$ for all $(t,x)\in B_R(t_0,x_0)\setminus\{(t_0,x_0)\}$ for some $R>0$, the function
\begin{equation}\label{auxiliarsub}
v(t,x)\coloneqq \begin{cases}
\phi(t,x) \text{ in } B_R(t_0,x_0);\\
u(t,x) \text{ at } (0,T]\times \mathbb{R}^n \setminus B_R(t_0,x_0)
\end{cases}
\end{equation}
satisfies $\pazocal{L}v (t_0,x_0)\leq f(t_0,x_0)$.

Analogously, a lower semicontinuous function $u$ on $(0,T]\times \mathbb{R}^n$ is a supersolution of $\pazocal{L}u=f$ at $(t_0,x_0)$  if for all functions $\varphi\in C^{1,2}(\overline{B_R(t_0,x_0)})$ such that $0=(u-\varphi)(t_0,x_0)< (u-\varphi)(t,x)$ for all $(t,x)\in B_R(t_0,x_0)\setminus\{(t_0,x_0)\}$ for some $R>0$, the function
\begin{equation}\label{auxiliarsuper}
v(t,x)\coloneqq \begin{cases}
\varphi(t,x) \text{ in } B_R(t_0,x_0);\\
u(t,x) \text{ at } (0,T]\times \mathbb{R}^n \setminus B_R(t_0,x_0)
\end{cases}
\end{equation}
satisfies $\pazocal{L}v (t_0,x_0)\geq f(t_0,x_0)$.

A solution of $\pazocal{L}u=f$ is a continuous function that is both a subsolution and a supersolution for all $(t,x)\in(0,T]\times\mathbb{R}^n$.
\end{definition}
Observe the definition of the auxiliar function $v$ is necessary due to the nonlocal nature of the operators $(-\Delta)^s$ and $\pazocal{I}$.

 We now give the notion of a continuous viscosity solution of the obstacle problem \eqref{prin}:

\begin{definition}\label{defviscosity} An upper semicontinuous, bounded function $u$ on $(0,T]\times \mathbb{R}^n$ is a subsolution of \eqref{prin} if $\pazocal{L}u(t,x)\leq 0$ in the viscosity sense, for all $(t,x)\in(0,T]\times\mathbb{R}^n$ such that $u(t,x)>\psi(x)$, and $u(0,\cdot)\leq \psi$.

Analogously, a lower semicontinuous, bounded function $u$ on $(0,T]\times \mathbb{R}^n$ is a supersolution of \eqref{prin} if $u(t,\cdot)\geq \psi$ for all $t\in(0,T]$, $\pazocal{L}u(t,x)\geq 0$ in viscosity sense for all $(t,x)\in(0,T]\times\mathbb{R}^n$, and $u(0,\cdot)\geq \psi$.

A solution of \eqref{prin} is a bounded continuous function that is both a subsolution and a supersolution.
\end{definition}
We remark that the definitions of subsolution and supersolution are not symmetric. Moreover, we can relax \Cref{defviscosity} by dropping the hypothesis $u(t,\cdot)\geq \psi$ (but assuming that a supersolution also has an empty semi-jet set, see \cite[Definition 2]{viscositydefinition} for the classical case).

Since the fractional Laplacian is the leading term, we now define a lower order operator $\pazocal{R}$ as 
\[
\pazocal{R}u(t,x)\coloneqq (\pazocal{I}+b\cdot \nabla +r)u(t,x) = \pazocal{I}u(t,x)+b\cdot \nabla u(t,x) +ru(t,x).
\] 
We now prove the existence, uniqueness and regularity of solutions of equation \eqref{eq:appendix} below (see \Cref{supremum}, \Cref{schauder}, and \Cref{regularityspacetime}). We need these tools in order to prove existence, uniqueness and regularity of solutions for a penalized equation, see \eqref{aprox}. Although the techniques are fairly standard, we were unable to find these results elsewhere, and we present details here for the sake of completeness.

\lemma[Uniqueness]\label{supremum} Assume $\psi$, $b$, $r$, and $\pazocal{I}$ satisfy \eqref{psi}, \eqref{defb}, \eqref{defr}, and \eqref{K}, respectively, and let $\alpha,\beta\in (0,1)$. For continuous functions $f\in L^\infty((0,T]\times\mathbb{R}^n)$ and $ u\in L^\infty((0,T]\times\mathbb{R}^n)\cap C^{1,2}_{t,x}((0,T]\times\mathbb{R}^n)$, which satisfy
\begin{equation}\label{eq:appendix}
\begin{split}
\partial_t u+(-\Delta)^su-\pazocal{R}u=f & \ \text{ in } \ (0,T]\times\mathbb{R}^n;\\
u(0,\cdot)=\psi & \quad \text{on } \mathbb{R}^n,
\end{split}
\end{equation} we have
\[
\|u\|_{L^\infty((0,T]\times\mathbb{R}^{n})}\leq C \Big( \|f\|_{L^\infty((0,T]\times\mathbb{R}^n)}+\|\psi\|_{L^\infty(\mathbb{R}^n)}\Big),
\]
where $C=C(r,T)$. In particular, the solution is unique.
\proof
It suffices to prove the first claim for functions $u$ for which
\[
\Big(\partial_t+(-\Delta)^s-M^+_{\mathcal{L}_0}-b\cdot\nabla-r\Big)u\leq f\leq \left(\partial_t+(-\Delta)^s-M^-_{\mathcal{L}_0}-b\cdot\nabla-r\right) u.
\]
The result follows by noticing that
\[\begin{split}
\left(\partial_t+(-\Delta)^s-M^-_{\mathcal{L}_0}-b\cdot\nabla+\gamma-r\right)\left(\pm e^{-\gamma t}u+\|f\|_{L^\infty(\mathbb{R}^n)} \right) & \geq \pm e^{-\gamma t}f+(\gamma-r)\|f\|_{L^\infty(\mathbb{R}^n)}\\
&\geq e^{-\gamma t}\big(\pm f + e^{\gamma t} \|f\|_{L^\infty(\mathbb{R}^n)}\big)\geq 0
\end{split}\]
where $\gamma\coloneqq r+1$. Thus, by the minimum principle,
\[
\pm e^{-\gamma t}u+\|f\|_{L^\infty(\mathbb{R}^n)}\geq -\max_{\{t=0\}\times \mathbb{R}^n}(\pm e^{-\gamma t}u+\|f\|_{L^\infty(\mathbb{R}^n)})^-=-\max_{\mathbb{R}^n}(\pm\psi+\|f\|_{L^\infty(\mathbb{R}^n)})^-,
\]
which gives
\[
\|u\|_{L^\infty((0,T]\times\mathbb{R}^{n})}\leq e^{\gamma T}(\|f\|_{L^\infty(\mathbb{R}^n)}+\max_{\mathbb{R}^n}(\pm\psi+\|f\|_{L^\infty(\mathbb{R}^n)})^-)\leq 2e^{\gamma T}(\|f\|_{L^\infty(\mathbb{R}^n)}+\|\psi\|_{L^\infty(\mathbb{R}^n)}).
\]
Now, if $u$ and $v$ satisfy \eqref{eq:appendix}, then
\begin{equation}\label{uniqueness}
\begin{split}
&(\partial_t+(-\Delta)^s-b\cdot\nabla-r)(u-v)+\pazocal{I}v-\pazocal{I}u=0;\\
&(u-v)(0,\cdot)= 0.
\end{split}
\end{equation}
By the ellipticity of $\pazocal{I}$ (see \eqref{propertyI}), we have
\[\begin{split}
&(\partial_t+(-\Delta)^s-M^-_{\mathcal{L}_0}-b\cdot\nabla+\gamma-r)e^{-\gamma t}(u-v)\geq 0\\
&(\partial_t+(-\Delta)^s-M^+_{\mathcal{L}_0}-b\cdot\nabla+\gamma-r)e^{-\gamma t}(u-v)\leq 0\\
&(u-v)(0,\cdot)= 0.
\end{split}\]
By minimum and maximum principles, respectively, we conclude $u\equiv v$.
\endproof

Next, we need a regularity result for the fractional heat equation. Namely, by \cite[Theorems 2.3 and 3.1]{higherparabolicreg}, if $v$ satisfies
\[
\partial_tv+(-\Delta)^sv=f
\]
with $f\in C^{\alpha,\beta}_{t,x}((0,T];\mathbb{R}^n)$, $\alpha,\beta\in(0,1)$, then
\begin{equation}\label{regalphabeta}
\|\partial_t v\|_{C^{\alpha,\beta}_{t,x}((0,T];\mathbb{R}^n)}+\|(-\Delta)^s v\|_{C^{\alpha,\beta}_{t,x}((0,T];\mathbb{R}^n)}\leq C\Big(1+\|f\|_{C^{\alpha,\beta}_{t,x}((0,T];\mathbb{R}^n)}\Big).
\end{equation}
Moreover, since $\pazocal{R}$ is a lower order operator with respect to $(-\Delta)^s$, we can perform an interpolation inequality. Recall that $\max\{1,2\sigma\}<2s$. Given  a bounded function $u$, by classical H\"{o}lder interpolation inequalities (see, for instance, \cite[Lemma 6.32]{GiTr-01} and/or \cite[Propositions 2.1.8 and 2.1.9]{silvestre}), we have that
\[\begin{split}
\|\pazocal{R}u\|_{C^{\alpha}(\mathbb{R}^n)}\leq \|u\|_{C^{\alpha+\max\{1,2\sigma\}}(\mathbb{R}^n)}&\leq \epsilon \|u\|_{C^{\alpha+2s}(\mathbb{R}^n)}+C_\epsilon\|u\|_{L^\infty(\mathbb{R}^n)}\\
&\leq \epsilon\|(-\Delta)^su\|_{C^{\alpha}(\mathbb{R}^n)}+C_\epsilon\|u\|_{L^\infty(\mathbb{R}^n)};\\
\|\pazocal{R}u\|_{L^\infty(\mathbb{R}^n)}\leq \|u\|_{C^{\max\{1,2\sigma+0^+\}}(\mathbb{R}^n)}&\leq \epsilon \|u\|_{C^{2s-0^+}(\mathbb{R}^n)}+C_\epsilon\|u\|_{L^\infty(\mathbb{R}^n)},
\end{split}\]
for all $\epsilon>0$  and $\alpha\in(0,1)$. We now prove a priori estimates for classical solutions of \eqref{eq:appendix} (see \cite[Lemma 2.6]{ellipticcase} for a proof in the elliptic case).
\lemma[A priori Schauder estimates]\label{schauder} Assume $\psi$, $b$, $r$, and $\pazocal{I}$ as in \eqref{psi}, \eqref{defb}, \eqref{defr}, and \eqref{K}, respectively. Then, there exists a constant $C(n,s,\lambda,\Lambda,\sigma,T,r,b)$ such that for any $f\in C^{\alpha,\beta}_{t,x}$ and $u \in C^{1+\alpha,2}_{t,x}$ bounded functions which satisfy
\[\begin{split}
\partial_t u +(-\Delta)^s u-\pazocal{R}u=f & \quad  \text{on } (0,T]\times\mathbb{R}^n,\\
u(0,\cdot)=\psi & \quad \text{on } \mathbb{R}^n,
\end{split}\]
we have the estimate
\[
\|\partial_t u\|_{C^{\alpha,\beta}((0,T]\times\mathbb{R}^n)}+\|(-\Delta)^s u\|_{C^{\alpha,\beta}((0,T]\times\mathbb{R}^n)}\leq C \Big( \|\psi\|_{C^2(\mathbb{R}^n)}+\|f\|_{C^{\alpha,\beta}((0,T]\times\mathbb{R}^n)}\Big).
\]
In particular, we have $u(t,\cdot)\in C^{2s+\beta}$.
\proof We first notice that $v\coloneqq u-\psi\in C^{1+\alpha,2}_{t,x}$ satisfies
\[\begin{split}
&(\partial_t+(-\Delta)^s-M^+_{\mathcal{L}_0}-b\cdot\nabla-r)v\leq \tilde{f}\leq (\partial_t+(-\Delta)^s-M^-_{\mathcal{L}_0}-b\cdot\nabla-r)v;\\
&v(0,\cdot)=0,
\end{split}\]
where $\tilde{f}\coloneqq f-((-\Delta)^s-\pazocal{R})\psi$. By \eqref{regalphabeta}, there exists a constant $C$ such that
\[
\|\partial_t v\|_{C^{\alpha,\beta}((0,T]\times\mathbb{R}^n)}+\|(-\Delta)^s v\|_{C^{\alpha,\beta}((0,T]\times\mathbb{R}^n)}\leq C(1+\|\partial_tv+(-\Delta)^sv\|_{C^{\alpha,\beta}((0,T]\times\mathbb{R}^n)}).
\]
By \Cref{supremum}, we have $\|v\|_{L^\infty((0,T]\times\mathbb{R}^n)}\leq C\|\tilde{f}\|_{L^\infty((0,T]\times\mathbb{R}^n)}$ for a constant $C=C(T,r)$. Moreover, by interpolation inequalities, for all $\epsilon>0$, there exists $C_\epsilon=C(\epsilon,n,s,\lambda,\Lambda,\sigma,T,r,b)$ such that
\[
(1-\epsilon)\left(\|\partial_t v\|_{C^{\alpha,\beta}((0,T]\times\mathbb{R}^n)}+\|(-\Delta)^s v\|_{C^{\alpha,\beta}((0,T]\times\mathbb{R}^n)}\right)\leq C_\epsilon\|\tilde{f}\|_{C^{\alpha,\beta}((0,T]\times\mathbb{R}^n)},
\]
hence the estimative follows by choosing $\epsilon=1/2$. By the regularity of the fractional Laplacian, see \cite{silvestre}, we have $u(t,\cdot)\in C^{2s+\beta}$.
\endproof
We use the previous results to prove the existence and uniqueness of solutions of \eqref{eq:appendix}.
\lemma\label{regularityspacetime} In the same setting as above, there exists a unique bounded solution $u \in C^{1+\alpha,2s+\beta}_{t,x}$ of \eqref{eq:appendix} for $\alpha\in (0,1)$ and $\beta\in (2-2s,1)$, with the bound
\begin{equation}\label{schauderpsi}
\|\partial_t u\|_{C^{\alpha,\beta}((0,T]\times\mathbb{R}^n)}+\|(-\Delta)^s u\|_{C^{\alpha,\beta}((0,T]\times\mathbb{R}^n)}\leq C\Big(\|f\|_{C^{\alpha,\beta}((0,T]\times\mathbb{R}^n)}+\|\psi\|_{C^2(\mathbb{R}^n)}\Big).
\end{equation}
\proof Uniqueness and boundedness follow from \Cref{supremum}. We first assume that $\psi\in C_c^{\infty}$ and $f\in C^{\infty}$, with compact support in space. We define the operator $\pazocal{L}_0$ as the fractional heat operator, that is, $\pazocal{L}_0\coloneqq \partial_t+(-\Delta)^s$. We claim that a solution of
\begin{equation}\label{toymodel}\begin{split}
&\pazocal{L}_0u=f \text{ on } (0,T]\times\mathbb{R}^n,\\
&u(0,\cdot)=\psi,
\end{split}
\end{equation}
is smooth with compact support in space. Indeed, denoting $\mathcal{F}u$ the Fourier transform of $u$ in space, we have
\[
u(t,x)\coloneqq \mathcal{F}^{-1}\left(e^{-|\xi|^{2s}t}\mathcal{F}\psi\right)+\mathcal{F}^{-1}\left(\int_0^t e^{-|\xi|^{2s}(t-s)}\mathcal{F}f (s, \xi) \, \dd s 
\right).
\]
By the regularity of $\psi$ and $f$, we obtain $u\in C^\infty$, with compact support in space, concluding the claim. By \Cref{supremum} and \Cref{schauder}, we obtain \eqref{schauderpsi} for $\pazocal{L}_0$.

Now, note that functions $f\in C^{\alpha,\beta}_{t,x}$ and $\psi\in C^2$ can be approximated by $\{f_k\}_{k\geq 0}\subset C^\infty$, with compact support in space, and $\{\psi_k\}_{k\geq 0}\subset C_c^\infty$, respectively. More precisely, we have $f_k\longrightarrow f$ and $\psi_k\longrightarrow \psi$ pointwise, and the sequences are uniformly bounded. Define $u_k \in C^{1+\alpha,2s+\beta}_{t,x}$ (vanishing as $|x|\longrightarrow \infty$) as a solution of
\[\begin{split}
&\pazocal{L}_0u_k=f_k \text{ on } (0,T]\times\mathbb{R}^n,\\
&u_k(0,\cdot)=\psi_k.
\end{split}\]
By the Arzelá-Ascoli Theorem, we obtain a subsequence $u_{k_j}\longrightarrow u$ in $ C^{1+\alpha,2s+\beta}_{t,x}$; thus, $u$ is a solution of \eqref{toymodel}. By assumptions \eqref{defb}, \eqref{defr}, and \eqref{K}, the operator $\pazocal{L}: C^{1+\alpha,2s+\beta}_{t,x}\longrightarrow C^{\alpha,\beta}_{t,x}$ is well defined. Now, we proceed by the continuity method: we write $\pazocal{L}_t= \pazocal{L}_0-t\pazocal{R}$. By \Cref{supremum}, $\pazocal{L}$ is an injective operator. Since we have proven that $\pazocal{L}_0$ is a surjective operator, we conclude that $\pazocal{L}_0$ is a bijective operator, and the inverse $\pazocal{L}_0^{-1}$ is well-defined. Hence, 
\[
\pazocal{L}_t u=f \iff u=\pazocal{L}^{-1}_0(f+t\pazocal{R}u)\eqqcolon \pazocal{S}_0u.
\]
If we show that $\pazocal{S}_0$ is a contraction map, we have that $\pazocal{L}$ is bijective and the claim will be proven. Indeed, by \Cref{supremum} and \Cref{schauder}, for any $u,\,v\in C^{1+\alpha,2s+\beta}_{t,x}$ such that $(u-v)(0,\cdot)\equiv 0$, we conclude
\[
\|\pazocal{S}_0u-\pazocal{S}_0v\|_{C^{1+\alpha,2s+\beta}((0,T]\times\mathbb{R}^n)}\leq t\,C\,\|\pazocal{R}u-\pazocal{R}v\|_{C^{\alpha,\beta}((0,T]\times\mathbb{R}^n)}
\]
where $C$ is a universal constant. Now, by the regularity of $u$ and $v$, we have
\begin{equation}\label{holderregofR}
t\,C\,\|\pazocal{R}u-\pazocal{R}v\|_{C^{\alpha,\beta}((0,T]\times\mathbb{R}^n)}\leq t\,C_0\,\|u-v\|_{C^{\alpha,\beta+\max\{1,2\sigma\}}((0,T]\times\mathbb{R}^n)},
\end{equation}
where $C_0$ does not depend on $t$. Since $\max\{1,2\sigma\}<2s$, we obtain $\pazocal{S}_0$ is a contraction map for $t_0<C_0^{-1}$. Hence, $\pazocal{L}_{t_0}$ is bijective, and the lemma follows by iterating the same argument for the map $\pazocal{S}_{t_0}u\coloneqq \pazocal{L}^{-1}_{t_0}(f+(t-t_0)\pazocal{R}u)$.
\endproof

Once the Hölder regularity of solutions of \eqref{eq:appendix} is established, we can prove the existence of solutions to the penalized equation 
\begin{equation}\label{aprox}
\begin{cases}
 u^\epsilon_{t} + (-\Delta)^{s} u^\epsilon - b \cdot \nabla u^\epsilon -\pazocal{I}u^\epsilon - ru^\epsilon= \beta_\epsilon( u^\epsilon - \psi^\epsilon) & \text{ in } (0,T] \times \mathbb{R}^{n}; \\
u^\epsilon(0,x)=\psi^\epsilon(x) & \text{ in } \mathbb{R}^{n}.
\end{cases}
\end{equation}
\lemma\label{regularityofuepsilon} Assume that $\psi$, $b$, $r$, and $\pazocal{I}$ as in \eqref{psi}, \eqref{defb}, \eqref{defr}, and \eqref{K}, respectively. Then, there exists a solution $u^\epsilon\in C^{1+\alpha,2s+\beta}_{t,x}$ to the penalized problem \eqref{aprox}, where $\alpha\in(0,1)$ and $\beta\in(2-2s,2)$.
\proof Let $\psi^\epsilon\coloneqq \eta_\epsilon\ast \psi$, where $\eta_\epsilon$ is the standard mollifier and $\epsilon>0$. 
We construct $u_k \in C^{1+\alpha,2s+\beta}_{t,x}$ iteratively as the unique solution to the equation
\begin{equation}\label{defuk}\begin{split}
&\pazocal{L} u_k= \beta_\epsilon(u_{k-1}-\psi^\epsilon) \quad \text{ on } (0,T]\times \mathbb{R}^n,\\
&u_k(0,\cdot)=\psi^\epsilon,
\end{split}
\end{equation}
where $u_0\equiv 0$. Indeed, for $k=1$, $u_1\in C^{1+\alpha,2s+\beta}_{t,x}$ since $\beta_\epsilon(-\psi^\epsilon)\in C^\infty$. Assuming the regularity holds for $u_{k-1}$, $u_k\in C^{1+\alpha,2s+\beta}_{t,x}$ since $\beta_\epsilon(u_{k-1}-\psi^\epsilon)\in C^{1+\alpha,2s+\beta}_{t,x}$ (by the regularity of $u_{k-1}$ and $\psi^\epsilon$).

By \eqref{schauderpsi} and the regularity above, we have for $k\geq 1$
\begin{equation}\label{mainregularityuk}
\begin{split}
\|\partial_t u_k\|_{C^{\alpha,\beta}((0,T]\times\mathbb{R}^n)}+&\|(-\Delta)^s u_k\|_{C^{\alpha,\beta}((0,T]\times\mathbb{R}^n)} \\
&\leq C\Big(\|\beta_\epsilon(u_{k-1}-\psi^\epsilon)\|_{C^{\alpha,\beta}((0,T]\times\mathbb{R}^n)}+\|\psi^\epsilon\|_{C^2(\mathbb{R}^n)}\Big).
\end{split}
\end{equation}

We now claim that, for $k\geq 1$
\begin{equation}\label{boundgammauk}
\begin{split}
&\|\beta_\epsilon(u_{k-1}-\psi^\epsilon)\|_{L^\infty((0,T]\times\mathbb{R}^n)}\leq C_\epsilon,\\
&\|\beta_\epsilon(u_{k-1}-\psi^\epsilon)\|_{C^{\alpha,\beta}((0,T]\times\mathbb{R}^n)}\leq C_\epsilon(1+\|u_{k-1}\|_{C^{\alpha,\beta}((0,T]\times\mathbb{R}^n)}),
\end{split}
\end{equation}
where $C_\epsilon$ depends on $\epsilon$ (but does not depend on $k$). Indeed, for $k=1$,
\[\begin{split}
&\|\beta_\epsilon(u_{0}-\psi^\epsilon)\|_{L^\infty((0,T]\times\mathbb{R}^n)}\leq e^{\epsilon^{-1}\|\psi^\epsilon\|_{L^\infty(\mathbb{R}^n)}}\leq C_\epsilon;\\
&\|\beta_\epsilon(u_{0}-\psi^\epsilon)\|_{C^{\alpha,\beta}((0,T]\times\mathbb{R}^n)}\le \frac{1}{\epsilon}\|\beta_\epsilon(u_{0}-\psi^\epsilon)\|_{L^\infty((0,T]\times\mathbb{R}^n)}\|\psi^\epsilon\|_{C^{\beta}(\mathbb{R}^n)}\leq C_\epsilon.
\end{split}\]
Now, suppose that \eqref{boundgammauk} holds for $k\geq 2$. Then by \Cref{supremum}, we have
\[
\|\beta_\epsilon(u_{k}-\psi^\epsilon)\|_{L^\infty((0,T]\times\mathbb{R}^n)}\leq e^{\epsilon^{-1}(\|\psi^\epsilon\|_{L^\infty(\mathbb{R}^n)}+\|\beta_\epsilon(u_{k-1}-\psi^\epsilon)\|_{L^\infty((0,T]\times\mathbb{R}^n)})}\leq C_\epsilon;\]
\[\begin{split}
\|\beta_\epsilon(u_{k}-\psi^\epsilon)\|_{C^{\alpha,\beta}((0,T]\times\mathbb{R}^n)}
&\leq \frac{1}{\epsilon}\|\beta_\epsilon(u_{k}-\psi^\epsilon)\|_{L^\infty((0,T]\times\mathbb{R}^n)}\|u_{k}-\psi^\epsilon\|_{C^{\alpha,\beta}((0,T]\times\mathbb{R}^n)}\\
&\leq C_\epsilon(1+\|u_{k}\|_{C^{\alpha,\beta}((0,T]\times\mathbb{R}^n)}).
\end{split}\]
Hence, the claim follows. Combining \eqref{mainregularityuk} and \eqref{boundgammauk}, we have
\[
\|\partial_t u_k\|_{C^{\alpha,\beta}((0,T]\times\mathbb{R}^n)}+\|(-\Delta)^s u_k\|_{C^{\alpha,\beta}((0,T]\times\mathbb{R}^n)}
\leq C_\epsilon(1+\|u_{k-1}\|_{C^{\alpha,\beta}((0,T]\times\mathbb{R}^n)}).
\]
We finally claim that
\begin{equation}\label{bounduk-1}
\|u_{k-1}\|_{C^{\alpha,\beta}((0,T]\times\mathbb{R}^n)}\leq C_\epsilon.
\end{equation}
Once the claim is proved, we have a uniform (with respect to $k$) bound of 
\[
\|\partial_t u_k\|_{C^{\alpha,\beta}((0,T]\times\mathbb{R}^n)}+\|(-\Delta)^s u_k\|_{C^{\alpha,\beta}((0,T]\times\mathbb{R}^n)}.
\]
The claim follows by the regularity of the fractional heat equation with a bounded source term (see \eqref{fbounded}):
\[\begin{split}
\|u_{k-1}\|_{C^{1-0^+}((0,T];L^\infty(\mathbb{R}^n))}+&\|u_{k-1}\|_{L^\infty((0,T];C^{2s-0^+}(\mathbb{R}^n))}\\
&\leq C\Big(\|(\partial_t+(-\Delta)^s) u_{k-1}\|_{L^\infty((0,T]\times\mathbb{R}^n)}+\|u_{k-1}\|_{L^\infty((0,T]\times\mathbb{R}^n)}\Big).
\end{split}\]
By interpolation inequalities (see, for instance, \cite[Lemma 6.32]{GiTr-01}) and \Cref{supremum}, we obtain, for a constant $C$ depending on $n,s,\lambda,\Lambda,\sigma,T,r,b,\|\psi\|_{C^2(\mathbb{R}^n)}$, that
\[
\|u_{k-1}\|_{C^{1-0^+}((0,T];L^\infty(\mathbb{R}^n))}+\|u_{k-1}\|_{L^\infty((0,T];C^{2s-0^+}(\mathbb{R}^n))}\leq C(1+\|\beta_\epsilon(u_{k-1}-\psi^\epsilon)\|_{L^\infty((0,T]\times\mathbb{R}^n)}).
\]
Hence, by \eqref{boundgammauk}, we conclude our claim.

Now, by the uniform bound of $\{u_k\}_{k\geq0}$ in $C^{1+\alpha,2s+\beta}_{t,x}$, we have a convergent subsequence in compact subsets of $(0,T]\times\mathbb{R}^n$ in $C^{1+\alpha,2s+\beta}_{t,x}$ to a function $u^\epsilon \in C^{1+\alpha,2s+\beta}_{t,x}((0,T]\times\mathbb{R}^n)$. Moreover, we have 
\[
\pazocal{L} u_k \rightarrow \pazocal{L} u_\epsilon \quad \text{and} \quad \beta(u_k-\psi^\epsilon)\rightarrow \beta_\epsilon(u^\epsilon-\psi^\epsilon) \quad \text{ as } k\rightarrow \infty.
\]
Hence, we conclude
\[\begin{split}
&\pazocal{L} u^\epsilon=\beta_\epsilon(u^\epsilon-\psi^\epsilon)\quad \text{on } (0,T]\times\mathbb{R}^n,\\
&u^\epsilon(0,\cdot)=\psi^\epsilon. \qedhere
\end{split}\]
\endproof
We now prove a uniform bound of $\beta_\epsilon(u^\epsilon-\psi^\epsilon)$ with respect to $\epsilon$, which combined with \Cref{supremum}, gives a uniform bound of $u^\epsilon$.
\lemma\label{uniformbound} Assume that $\psi$, $b$, $r$, and $\pazocal{I}$ as in \eqref{psi}, \eqref{defb}, \eqref{defr}, and \eqref{K}, respectively. Then, there exists a constant $C = C(n,s,\sigma,\lambda,\Lambda,T,\|\psi\|_{C^2(\mathbb{R}^n)},b,r)>0$ such that
\begin{equation}\label{boundbeta}\begin{split}
&\|\beta_\epsilon(u^\epsilon-\psi^\epsilon)\|_{L^\infty((0,T]\times\mathbb{R}^n)}\leq C\\
&\|u^\epsilon\|_{L^\infty((0,T]\times\mathbb{R}^n)}\leq C.
\end{split}
\end{equation}
\proof
We remark that we only need an upper bound, since $\beta_\epsilon(u^\epsilon-\psi^\epsilon)\geq0$. We assume for $\gamma\geq 0$
\[
\inf_{(0,T]\times \mathbb{R}^n}(e^{-\gamma t} u^\epsilon - \psi^\epsilon)<0,
\] for otherwise $u^\epsilon\geq e^{\gamma t}\psi^\epsilon\geq \psi^\epsilon$, hence $\beta_\epsilon(u^\epsilon-\psi^\epsilon)\leq 1$. Take $\varphi$ a nonnegative smooth function that grows as $|x|^\sigma$ at infinity. Now, we claim that for $\delta>0$ sufficiently small, we have
\begin{equation}\label{negativeminimum}
\min_{(0,T]\times \mathbb{R}^n}\left(e^{-\gamma t} u^\epsilon-\psi^\epsilon+\frac{\delta}{T-t}+\delta \varphi\right)<0,
\end{equation}
and the minimum is a interior point of $(0,T]\times \mathbb{R}^n$. Indeed, since $u^\epsilon$ is bounded\footnote{By \Cref{supremum} and taking the limit $k\rightarrow \infty$ at \eqref{boundgammauk}, we conclude $\|u^\epsilon\|_{L^\infty((0,T]\times \mathbb{R}^n)}\leq C_\epsilon$. However, we do not have a uniform bound with respect to $\epsilon$.}, we may take $\delta$ small enough so that we can consider $(t^\epsilon_\delta,x^\epsilon_\delta)$ the minimizer of $e^{-\gamma t} u^\epsilon-\psi^\epsilon+\frac{\delta}{T-t}+\delta \varphi$. Now, since $\inf_{(0,T]\times \mathbb{R}^n}(e^{-\gamma t} u^\epsilon - \psi^\epsilon)<0$, we may assume \eqref{negativeminimum} (taking $\delta$ smaller if necessary). To prove that the minimizer is at the interior, we remark that the function blows up as $|x|\rightarrow \infty$ and $t\rightarrow T^-$. Moreover, if the minimum were at $t=0$, then
\[
0<\delta\left(\frac{1}{T}+\inf_{\mathbb{R}^n}\varphi\right)=\inf_{(0,T]\times \mathbb{R}^n}\left(e^{-\gamma t} u^\epsilon-\psi^\epsilon+\frac{\delta}{T-t}+\delta \varphi\right)<0.
\] 
Hence, the claim is proven. Thus,
 \[
\partial_tu^\epsilon(t^\epsilon_\delta,x^\epsilon_\delta)-\gamma u^\epsilon(t^\epsilon_\delta,x^\epsilon_\delta)+\frac{\delta e^{\gamma t_\delta^\epsilon}}{(T-t^\epsilon_\delta)^2}=0, \qquad \nabla u^\epsilon(t^\epsilon_\delta,x^\epsilon_\delta)-e^{\gamma t_\delta^\epsilon}\nabla\psi^\epsilon(x^\epsilon_\delta)+e^{\gamma t_\delta^\epsilon}\delta\nabla\varphi(x^\epsilon_\delta)=0,
\]
\[
\text{ and } \quad (-\Delta)^s u^\epsilon(t^\epsilon_\delta,x^\epsilon_\delta)-e^{\gamma t_\delta^\epsilon}(-\Delta)^s\psi^\epsilon(x^\epsilon_\delta)+\delta \;e^{\gamma t_\delta^\epsilon}(-\Delta)^s\varphi(x^\epsilon_\delta)\leq0.
\]
Furthermore, by \eqref{propertyI}, we have
\[
-\pazocal{I}u^\epsilon(t^\epsilon_\delta,x^\epsilon_\delta)\leq -M^-_{\mathcal{L}_0} u^\epsilon(t^\epsilon_\delta,x^\epsilon_\delta)\leq e^{\gamma t_\delta^\epsilon}(\delta M^+_{\mathcal{L}_0}\varphi(x^\epsilon_\delta)-M^-_{\mathcal{L}_0}\psi^\epsilon(x^\epsilon_\delta)).
\]
Hence, choosing $\gamma\coloneqq r$, we have
\[\begin{split}
\beta_\epsilon( u^\epsilon - \psi^\epsilon)(t^\epsilon_\delta,x^\epsilon_\delta) & \leq e^{\gamma t_\delta^\epsilon}\Bigg( -\frac{\delta }{(T-t^\epsilon_\delta)^2}- b\cdot\nabla\psi^\epsilon(x^\epsilon_\delta)+\delta\;b\cdot\nabla\varphi(x^\epsilon_\delta)+(-\Delta)^s\psi^\epsilon(x^\epsilon_\delta)\\
&-\delta \;(-\Delta)^s\varphi(x^\epsilon_\delta)+\delta\; M^+_{\mathcal{L}_0}\varphi(x^\epsilon_\delta)-M^-_{\mathcal{L}_0}\psi^\epsilon(x^\epsilon_\delta)+(\gamma-r)\psi^\epsilon(x^\epsilon_\delta)\\
&-(\gamma-r) \frac{\delta}{(T-t^\epsilon_\delta)^2}-\delta(\gamma-r) \varphi(x^\epsilon_\delta) \Bigg)\le C+O(\delta),
\end{split}\]
where $C(n,s,\sigma,\lambda,\Lambda,T,b,r,\|\psi\|_{C^2(\mathbb{R}^n)})>0$. Since \[(u^\epsilon - \psi^\epsilon)(t^\epsilon_\delta,x^\epsilon_\delta)\longrightarrow \inf_{(0,T]\times\mathbb{R}^n}(u^\epsilon - \psi^\epsilon)\] as $\delta\rightarrow 0$ and $\beta_\epsilon$ is decreasing, we obtain 
\[
\sup_{(0,T]\times\mathbb{R}^n}\beta_\epsilon(u^\epsilon - \psi^\epsilon)  =\lim_{\delta\rightarrow 0}\beta_\epsilon(u^\epsilon - \psi^\epsilon)(t^\epsilon_\delta,x^\epsilon_\delta) \leq C\|\psi\|_{C^2(\mathbb{R}^n)}.
\]
For the second inequality in \eqref{boundbeta}, by \Cref{supremum} and the previous result, we conclude
\[
\|u^\epsilon\|_{L^\infty((0,T]\times\mathbb{R}^n)}\leq C(\|\beta_\epsilon(u^\epsilon-\psi^\epsilon)\|_{L^\infty((0,T]\times\mathbb{R}^n)}+\|\psi\|_{L^\infty(\mathbb{R}^n)})\leq C\|\psi\|_{C^2(\mathbb{R}^n)}. \qedhere
\]
\endproof
We finally prove that $u^\epsilon$, a solution of \eqref{aprox}, converges to $u$, a solution of \eqref{prin}.
\theorem[Approximation by Penalization Method]\label{existenceproof}  Assume that $\psi$, $b$, $r$, and $\pazocal{I}$ as in \eqref{psi}, \eqref{defb}, \eqref{defr}, and \eqref{K}, respectively. Then, there exists a viscosity solution of \eqref{prin} which is an approximation of a solution $u^\epsilon$ of \eqref{aprox}, i.e., $u^\epsilon\longrightarrow u$ as $\epsilon\longrightarrow 0^+$, and $u\in C^{1-0^+}((0,T];L^\infty(\mathbb{R}^n))\cap L^\infty((0,T];C^{2s-0^+}(\mathbb{R}^n))$.
\proof We know by \Cref{regularityofuepsilon} that, for each $\epsilon>0$, $u^\epsilon$ is a $C^{1+\alpha,2s+\beta}_{t,x}$ function. By \eqref{fbounded} and interpolation inequalities (see, for instance, \cite[Lemma 6.32]{GiTr-01}), we have
\[
\|u^\epsilon\|_{C^{1-0^+}((0,T];L^\infty(\mathbb{R}^n))}+\|u^\epsilon\|_{L^\infty((0,T];C^{2s-0^+}(\mathbb{R}^n))}\leq C(\|\beta_\epsilon(u^\epsilon-\psi^\epsilon)\|_{L^\infty((0,T]\times\mathbb{R}^n)}+\|u^\epsilon\|_{L^\infty((0,T]\times\mathbb{R}^n)}).
\]
By \Cref{uniformbound}, we conclude that $u^\epsilon\longrightarrow u$ in both $C^{1-0^+}_t;L^\infty_x$ and
$L^\infty_t;C^{2s-0^+}_x$ norms as $\epsilon\longrightarrow 0^+$, and $u\in C^{1-0^+}((0,T];L^\infty(\mathbb{R}^n))\cap L^\infty((0,T];C^{2s-0^+}(\mathbb{R}^n))$. Moreover,  $\psi^\epsilon\longrightarrow \psi$ in $C^2$ norm. To show that in fact $u$ is a viscosity solution of \eqref{prin}, let $\phi \in C^{1,2}_{t,x}$ such that $u-\phi$ has a strict local  maximum  at $(t_0,x_0)$. Choose $R>0$ such that $0=(u-\phi)(t_0,x_0)> (u-\phi)(t,x)$ for $(t,x)\in B_R(t_0,x_0)\setminus\{(t_0,x_0)\}$ and consider $(t^\epsilon,x^\epsilon)$ the maximum of $u^\epsilon-\phi$ at $K\coloneqq B_{R/2}(t_0,x_0)$. By the compactness of $K$, we have (up to a subsequence) $(t^\epsilon,x^\epsilon)\longrightarrow (s,y)$ as $\epsilon\longrightarrow 0^+$. By the definition of $(t^\epsilon,x^\epsilon)$ and its limit, we have
\[
(u-\phi)(t_0,x_0)\leq (u-\phi)(s,y).
\]
Since $(t_0,x_0)$ is  a strict local maximum, $(s,y)=(t_0,x_0)$. Now, since $u^\epsilon$ solves \eqref{aprox} classically and the definition of $(t^\epsilon,x^\epsilon)$, we have
\[
(\partial_t v^\epsilon+(-\Delta)^s v^\epsilon-\pazocal{I} v^\epsilon-b\cdot\nabla v^\epsilon-ru^\epsilon)(t^\epsilon,x^\epsilon)\leq\beta_\epsilon(u^\epsilon-\psi^\epsilon)(t^\epsilon,x^\epsilon),
\]
where $v^\epsilon$ as in \eqref{auxiliarsub}, replacing $u$ by $u^\epsilon$. By the uniform bound \eqref{boundbeta}, we have that $u(t,\cdot)\geq \psi$ for all $t\in(0,T]$. By letting $\epsilon\longrightarrow 0^+$ at the above inequality, we conclude
\[
(\partial_t v+(-\Delta)^sv-\pazocal{I}v-b\cdot\nabla v-rv)(t_0,x_0)\leq\begin{cases} 1, \quad \text{if } v(t_0,x_0)=\psi(x_0);\\
0, \quad \text{if } v(t_0,x_0)>\psi(x_0).
\end{cases}
\]
Hence,
\[\begin{cases}
\partial_t v+(-\Delta)^sv-\pazocal{I}v-b\cdot\nabla v-r v)(t_0,x_0)\leq 0 & \text{if } u(t_0,x_0)>\psi(x_0); \\
u(0,x)  = \psi(x) & \forall \; x\in \mathbb{R}^{n}.
\end{cases}\]
Since $(t_0,x_0)$ is arbitrary, we conclude that $u$ is a viscosity subsolution. To show that $u$ is also a viscosity supersolution, we remark that by the same ideas as above that for $0=(u-\varphi)(t_0,x_0)< (u-\varphi)(t,x)$ for $(t,x)\in B_R(t_0,x_0)\setminus\{(t_0,x_0)\}$,  we obtain that
\[
(\partial_t v+(-\Delta)^sv-\pazocal{I}v-b\cdot\nabla v-r v)(t_0,x_0)\geq\begin{cases} 1, \quad &\text{if } v(t_0,x_0)=\psi(x_0);\\
0, \quad &\text{if } v(t_0,x_0)>\psi(x_0),
\end{cases}
\]
where $v$ as in \eqref{auxiliarsuper}. Thus,
\[\begin{cases}
\partial_t v+(-\Delta)^sv-\pazocal{I}v-b\cdot\nabla v-r v)(t_0,x_0)\geq 0; \\
u(0,\cdot)  = \psi,
\end{cases}\]
and it follows that $u$ is a viscosity supersolution, and hence $u$ is a viscosity solution of \eqref{prin}.
\endproof

We now want to establish the uniqueness of solutions of \eqref{prin}. In order to do so, we introduce, in the parabolic case, the sup-convolution and inf-convolution and the $\Gamma$-convergence; see \cite[Section 5]{ellipticity} for the elliptic case.
\begin{definition} Given an upper semicontinuous function $u$, the sup-convolution approximation $u^\epsilon$ is given by
\[
u^\epsilon(t,x)=\sup_{(s,y)\in (0,T]\times\mathbb{R}^n} \left\{u(t+s,x+y)-\frac{|(s,y)|^2}{\epsilon} \right\}.
\]
Analogously, if $u$ is lower semicontinuous, its inf-convolution $u_\epsilon$ is given by
\[
u_\epsilon(t,x)=\inf_{(s,y)\in(0,T]\times\mathbb{R}^n}\left\{u(t+s,x+y)+\frac{|(s,y)|^2}{\epsilon}\right\}.
\]
\end{definition}
Observe that $u$ bounded implies $u^\epsilon$ and $u_\epsilon$ bounded.
\begin{definition} A sequence of lower semicontinuous functions $u_k$ $\Gamma$-converges to $u$ in $(0,T]\times\mathbb{R}^n$ if the following two  conditions hold:
\begin{itemize}
\item For every sequence $(t_k,x_k)\longrightarrow (t,x)$, $\liminf_{k\rightarrow \infty} u_k(t_k,x_k)\geq u(t,x)$.
\item For every $(t,x)$, there exists a sequence $(t_k,x_k)\longrightarrow (t,x)$ such that $\limsup_{k\rightarrow \infty} u_k(t_k,x_k)= u(t,x)$.
\end{itemize}
\end{definition}
In the next proposition, we show that we are allowed to change the set of test functions $\phi$ in the \Cref{auxdefviscosity} of a viscosity solution by the set of functions that touch from above (below) and that are punctually $C^{1;1,1}_{t,x}$.
\proposition\label{changetestfunction} Let $u$ be an upper semicontinuous function such that $\pazocal{L}u\leq f$ in the viscosity sense. Let $\phi$ be a bounded function such that $\phi\in C^{1;1,1}_{t,x}$ punctually at $(t,x)$. Assume that $\phi$ touches $u$ from above at $(t,x)$. Then $\pazocal{L}\phi(t,x)$ is defined in the classical sense and $\pazocal{L}\phi(t,x)\leq f(t,x)$.
\proof By the assumed regularity of $\phi$, $\pazocal{L}\phi(t,x)$ is classically defined. Moreover, there exists a polynomial $q$, quadratic in space and linear in time, that touches $\phi$ from above at $(t,x)$. Let
\[
v_r\coloneqq \begin{cases} \phi \quad \text{in } B_r(t,x),\\
u \quad \text{in } (0,T]\times\mathbb{R}^n\setminus B_r(t,x).
\end{cases}
\]
Since $\pazocal{L}u\leq f$ in the viscosity sense, $\pazocal{L}v_r(t,x)\leq f(t,x)$, and $\pazocal{L}v_r(t,x)$ is well-defined. Let
\[
u_r\coloneqq \begin{cases} q \quad \text{in } B_r(t,x),\\
\phi \quad \text{in } (0,T]\times\mathbb{R}^n\setminus B_r(t,x).
\end{cases}
\]
Thus, we have
\[\begin{split}
\pazocal{L}\phi(t,x)&\leq \pazocal{L}u_r(t,x)+(M^+_{\mathcal{L}_0}-(-\Delta)^s)(u_r-\phi)(t,x)\\
&\leq \pazocal{L}u_r(t,x)\\
&\leq \pazocal{L}v_r(t,x)+(M^+_{\mathcal{L}_0}-(-\Delta)^s)(v_r-u_r)(t,x)\\
&\leq f(t,x)+\Lambda\int_{B_r(x)}\frac{(\delta(\phi-q)(x,y,t))^+}{|y|^{n+2\sigma}}\dd y+\int_{B_r(x)}\frac{(\delta(\phi-q)(x,y,t))}{|y|^{n+2s}}\dd y\\
&\leq f(t,x)+\epsilon,
\end{split}\]
for any $\epsilon>0$, since both integrands are bounded by $|y|^{2-2\sigma-n}$. The proposition follows.
\endproof
Of course, an analogue of \Cref{changetestfunction} holds for supersolutions, and its proof is similar. Analogously to \cite[Propositions 5.4 and 5.5]{ellipticity}, in our setting we have the following:
\proposition\label{toolcomparison} If $u$ is bounded and lower-semicontinuous in $(0,T]\times\mathbb{R}^n$, then $u_\epsilon$ $\Gamma$-converges to $u$. Likewise, if $u$ is bounded and upper-semicontinuous in $(0,T]\times\mathbb{R}^n$, then $-u^\epsilon$ $\Gamma$-converges to $-u$. If $u$ satisfies $\pazocal{L}u\leq f$ in the viscosity sense, then $\pazocal{L}u^\epsilon\leq f+d_\epsilon$ in the viscosity sense; if $v$ satisfies $\pazocal{L}v\geq f$ in the viscosity sense, then $\pazocal{L}v_\epsilon\geq f-d_\epsilon$ in the viscosity sense, where $d_\epsilon\longrightarrow 0$ as $\epsilon\longrightarrow 0$ and depends on the modulus of continuity.
\proof
The first claim is just a generalization $u^\epsilon\longrightarrow u$ locally uniformly if $u$ is continuous. For the second claim, suppose that the $f$ has modulus of continuity $\omega$. Let $(t_0,x_0)$ be such that $u^\epsilon(t,x)-\phi(t,x)< u^\epsilon(t_0,x_0)-\phi(t_0,x_0)=0$ for all $(t,x)\in B_R(t_0,x_0)\setminus \{(t_0,x_0)\}$, $\phi\in C^{1,2}_{t,x}$. Define $$\eta(s,y)\coloneqq \phi(s-s_0+t_0,y-y_0+x_0)+\frac{|(s_0-t_0,y_0-x_0)|^2}{\epsilon},$$ where $(s_0,y_0)$ is such that
\[
u^\epsilon(t_0,x_0)=u(s_0,y_0)-\frac{|(s_0-t_0,y_0-x_0)|^2}{\epsilon}.
\]
Then, $u(s,y)-\eta(s,y)< u(s_0,y_0)-\eta(s_0,y_0)=0$ for all $(s,y)\in B_R(s_0,y_0)\setminus \{(s_0,y_0)\}$, and so
\[
\pazocal{L}v^\epsilon(t_0,x_0)=\pazocal{L}v(s_0,y_0)\leq f(s_0,y_0),
\]
where $v$ is as in \eqref{auxiliarsub}, and $v^\epsilon$ as in \eqref{auxiliarsub}, replacing $u$ by $u^\epsilon$. Since $f$ has a modulus of continuity $\omega$, we have $f(s_0,y_0)\leq f(t_0,x_0)+\omega(|(t_0-s_0,x_0-y_0)|)$. Noticing that $u^\epsilon\geq u$, one has 
\[\frac{|(s_0-t_0,y_0-x_0)|^2}{\epsilon}\leq u(s_0,y_0)-u(t_0,x_0)\leq 2\|u\|_{L^\infty((0,T]\times\mathbb{R}^n)},\]
and we conclude
\[
\pazocal{L}v^\epsilon(t_0,x_0)\leq f(t_0,x_0)+d_\epsilon,
\]
where $d_\epsilon\coloneqq \omega(2\|u\|_{L^\infty((0,T]\times\mathbb{R}^n)}\epsilon^{1/2})$. The proof for supersolutions is analogous.
\endproof
The next lemma is a straightforward adaptation of \cite[Lemma 5.8]{ellipticity}, since the main difficulty of the operator $\pazocal{L}$ is the nonlocal part $(-\Delta)^s-\pazocal{I}$. 

\lemma\label{tool2comparison} Let $u$ and $v$ be bounded functions such that $u$ is upper-semicontinuous with $\pazocal{L}u\leq f$ in the viscosity sense, and $v$ is lower-semicontinuous with $\pazocal{L}v\geq g$  in the viscosity sense. Then
\[
\pazocal{L}^+(u-v)\coloneqq(\partial_t+(-\Delta)^s-b\cdot\nabla-r -M^+_{\mathcal{L}_0})(u-v)\leq f-g \text{ in the viscosity sense.}
\]

\proof
By \Cref{toolcomparison} and the stability of viscosity solutions under $\Gamma$-limits (see \cite[Lemma 4.5]{ellipticity}), it is enough to show that $\pazocal{L}^+(u^\epsilon-v_\epsilon)\leq  f-g+2d_\epsilon$ in the viscosity sense for every $\epsilon>0$. Let $\phi\in C^{1,2}_{t,x}$ touching from above $u^\epsilon-v_\epsilon$ at $(t,x)$. Since $u$ and $v$ are bounded, then $u^\epsilon$ and $v_\epsilon$ are also bounded. Since $u^\epsilon-v_\epsilon$ is touched by above at $(t,x)$ by a $C^{1,2}_{t,x}$ function, then both $u^\epsilon$ and $-v_\epsilon$ must be $C^{1,2}_{t,x}$ punctually at $(t,x)$. Moreover, by \eqref{changetestfunction}, we can evaluate $\pazocal{L}u^\epsilon$ and $\pazocal{L}v_\epsilon$ at $(t,x)$ in the classical sense. Thus, by \Cref{toolcomparison},
\[
\pazocal{L}^+(u^\epsilon-v_\epsilon)(t,x)\leq \pazocal{L}u^\epsilon(t,x)-\pazocal{L}v_\epsilon(t,x)\leq f(t,x)-g(t,x)+2d_\epsilon.
\]
Hence, $\pazocal{L}^+\phi(t,x)\leq f(t,x)-g(t,x)+2d_\epsilon$ since $\phi$ touches $v_\epsilon-u^\epsilon$ by above. Thus,
$\pazocal{L}^+(u^\epsilon-v_\epsilon)\leq  f-g+2d_\epsilon$ in the viscosity sense.
\endproof
We now prove a maximum principle for $\pazocal{L}^+ u\leq f$. This is the key result for our comparison principle of \Cref{comparisonprinciple} below.
\lemma\label{maximumviscosity} Let $u$ be a bounded upper-semicontinuous function defined in $(0,T]\times\mathbb{R}^n$ such that, in the viscosity sense, $\pazocal{L}^+ u\leq f$ in an open set $\Omega\subset(0,T]\times\mathbb{R}^n$. Then, there exists a constant $C = C(T)>0$ such that
\[\sup_{\Omega}u\leq C\left(\|f^+\|_{L^\infty((0,T]\times\mathbb{R}^n)}+\sup_{\Omega^c}u\right).\]
\proof For $t\in(0,T]$, set
\[
\phi_M(t)\coloneqq e^{\gamma t}(M+\epsilon+\|f^+\|_{L^\infty((0,T]\times\mathbb{R}^n)}),
\]
where $\gamma\coloneqq r+1$, and $\epsilon>0$. Note that $\pazocal{L}^+\phi_M(t,x)=(\gamma-r)\phi_M(t)>\|f^+\|_{L^\infty((0,T]\times\mathbb{R}^n)}$. Let $M_0$ be the smallest value of $M$ for which $\phi_M\geq u$ in $(0,T]\times\mathbb{R}^n$. We assume by contradiction that $M_0>\sup_{\Omega^c}u$. Then, there exists $(t_0,x_0)\in \Omega$ such that $u(t_0,x_0)=\phi_{M_0}(t_0)$ (by the minimality of $M_0$), and so $\phi$ touches $u$ from above at $(t_0,x_0)$. Since $u$ is a viscosity subsolution at $\Omega$, we would have $\pazocal{L}^+\phi_{M_0}(t_0,x_0)\leq f(t_0,x_0)$, a contradiction. Therefore, for $(t,x)\in(0,T]\times\mathbb{R}^n$, we have
\[\begin{split}
u(t,x)\leq \phi_{M_0}(t)&\leq e^{\gamma T}(M_0+ \epsilon+\|f^+\|_{L^\infty((0,T]\times\mathbb{R}^n)})\\
&\leq e^{\gamma T}\left(\sup_{\Omega^c}u+\epsilon+\|f^+\|_{L^\infty((0,T]\times\mathbb{R}^n)}\right).
\end{split}\]
Letting $\epsilon\longrightarrow 0$, we conclude the proof.
\endproof
We now prove the comparison principle for \eqref{prin}:
\theorem[Comparison Principle]\label{comparisonprinciple} Let $u,v$ be bounded viscosity subsolution and supersolution of \eqref{prin}, respectively. Then $u\leq v$ in $(0,T]\times \mathbb{R}^n$.
\proof
We first notice that $u(0,\cdot)\leq \psi\leq v(0,\cdot)$. For $t>0$, if $u(t,x)\leq\psi(x)$, then $v(t,x)\geq \psi(x)\geq u(t,x)$. Moreover, by \Cref{tool2comparison} we have $\pazocal{L}^+(u-v)\leq 0$ in the viscosity sense at $\{u>\psi\}$. By \Cref{maximumviscosity}, we conclude $u\leq v$.
\endproof
As a direct consequence (combined with \Cref{existenceproof}), we have the following corollary.
\begin{corollary}
[Existence, Uniqueness and Regularity]\label{mainthm1} There exists a unique bounded viscosity solution $u$ of \eqref{prin}. Moreover, $u\in C^{1-0^+,2s-0^+}_{t,x}$ and $u$ is approximated by a solution of the penalized equation \eqref{aprox}.  
\end{corollary} 
%\proof
%The uniqueness follows from \Cref{comparisonprinciple}. The existence, regularity and approximation follows from \Cref{existenceproof}.
%\endproof
Once the Comparison Principle is established, we are able to adapt preliminary regularity properties of solutions analogous to \cite[Lemma 3.2]{CaFi-08}. We implicitly use that $\pazocal{L}$ is translation invariant, since $b$ and $r$ are fixed and $\pazocal{I}$ is assumed to be translation invariant; see assumption \eqref{K}.

\begin{lemma}\label{lem:semiconvexity}
Let $u$ be a solution of \eqref{prin}. Then, for any fixed $t > 0$, $u(t, \cdot)$ is globally Lipschitz and uniformly semiconvex. Moreover, for any fixed $x \in \mathbb{R}^{n}$, the function $t \longmapsto u(t,x)$ non-decreasing.
\end{lemma}

\begin{proof} Fix $v\in\mathbb{R}^n$ and define $\tilde{u}(t,x)\coloneqq u(t,x+v)+C|v|$. Hence, $\tilde{u}$ solves
\[
\begin{cases}
\min \{\pazocal{L}\tilde{u}+rC|v|, \ \tilde{u}- \tilde{\psi}\} = 0 & \text{ in } (0,T] \times \mathbb{R}^{n}, \\
\tilde{u}(0,x)  = \tilde{\psi} & \text{ in } \mathbb{R}^{n},
\end{cases}
\]
where $\tilde{\psi}(x)\coloneqq \psi(x+v)+C|v|$. Choosing $C\coloneqq\|\nabla\psi\|_{L^\infty(\mathbb{R}^n)}$, if $u(t,x)=\psi(x)$, then $u(t,x)\leq \tilde{\psi}(x)\leq \tilde{u}(t,x)$. If $u(t,x)>\psi(x)$, then by \Cref{maximumviscosity}, we have $\pazocal{L}^+(u-\tilde{u})(t,x)\leq C|v|$, and by \Cref{comparisonprinciple}, $(u-\tilde{u})(t,x)\leq Cre^{\gamma T}|v|$, hence $u(t,\cdot)$ is globally Lipschitz\footnote{Notice that, by \Cref{existenceproof}, we already know $u(t,\cdot)$ is globally Lipschitz, but we have improved its Lipschitz constant with \Cref{lem:semiconvexity}.}.

Moreover, for any fixed $\eta \ge 0$, the function $\tilde{u}(t,x)\coloneqq u(t+\eta,x)$ solves 
\[
\begin{cases}
\min \{\pazocal{L}\tilde{u}, \ \tilde{u}- \psi\} = 0 & \text{ in } (-\eta,T-\eta] \times \mathbb{R}^{n}, \\
\tilde{u}(0,x)  = u(\eta,x) & \text{ in } \mathbb{R}^{n}.
\end{cases}
\]
 We know $u(t,x)\geq \psi(x)$ so that, in particular, $u(\eta,x) \geq \psi(x)$; therefore, by \Cref{comparisonprinciple},
\[
u(t+\eta,x)\geq u(t,x) \ \text{ for every } \eta,t\geq 0.
\]

Finally, denoting $C\coloneqq 2\,\|D^2\psi\|_{L^\infty(\mathbb{R}^n)}$, for a fixed $v\in \mathbb{R}^n$, we have
\[
\tilde{u}(t,x)\coloneqq\frac{u(t,x+v)+u(t,x-v)+C|v|^2}{2}\geq\frac{\psi(x+v)+\psi(x-v)+C|v|^2}{2}\geq \psi(x).
\]
If $u(t,x)=\psi(x)$, then $u$ is semiconvex. Moreover, since $\pazocal{I}$ is convex, $\tilde{u}$ satisfies
\[
\pazocal{L}\tilde{u}(t,x)\geq\frac{1}{2}(\pazocal{L}u(t,x+v)+\pazocal{L}u(t,x-v)-rC|v|^2)\geq -\frac{rC}{2}|v|^2.
\]
Hence, if $u(t,x)>\psi(x)$, then by \Cref{maximumviscosity}, $\pazocal{L}^+(u-\tilde{u})(t,x)\leq rC/2|v|^2$, and by \Cref{comparisonprinciple}, $(u-\tilde{u})(t,x)\leq Cre^{\gamma T}/2|v|^2$. Since $x,v$ are arbitrary, the $C_0$-semiconvexity of $u(t,\cdot)$ follows, where $C_0\coloneqq \|D^2\psi\|_{L^\infty(\mathbb{R}^n)}(1+re^{\gamma T})$.
\end{proof}

Our next lemma deals with basic estimates of our parabolic operator, which gives a Lipschitz regularity in spacetime (see \Cref{cor:lipschitz}) and a comparison between $(-\Delta)^su$ and $\pazocal{R}u$ at the contact set $\{u(t,\cdot)=\psi\}$ and the open set $\{u(t,\cdot)>\psi\}$ (see \Cref{comparison fractional and R}).

As a direct consequence, we have the following corollary.

\begin{corollary}\label{cor:lipschitz} 
If $u$ solves \eqref{prin}, then $u$ is Lipschitz in space-time with
\[
\|\partial_t u\|_{L^\infty((0,T]\times\mathbb{R}^n)}+\|\nabla u\|_{L^\infty((0,T]\times\mathbb{R}^n)}\le C(s,n,\lambda,\Lambda,r,b,T,\|\psi\|_{C^2(\mathbb{R}^n)}).
\]
\end{corollary}

\begin{proof} The Lipschitz regularity in space is just a restatement of \Cref{lem:semiconvexity}. Now, since $u$ solves \eqref{prin}, by \Cref{mainthm1}, $u$ is the limit of solutions $u^\epsilon$ of \eqref{aprox}. We denote by $\partial_t^h$ the difference quotient with respect to the time variable. Hence, $w^\epsilon(t,x)\coloneqq e^{-r t}\partial^h_t u^\epsilon(t,x)$ solves
\[\begin{cases}
 \partial_tw^\epsilon + (-\Delta)^{s} w^\epsilon - b \cdot \nabla w^\epsilon -  M^+_{\mathcal{L}_0}w^\epsilon + \epsilon^{-1}\beta_\epsilon(\xi)w^\epsilon\leq 0 & \text{ in } (0,T] \times \mathbb{R}^{n},\\
 |w^\epsilon(0,\cdot)|\leq |(b \cdot \nabla  + \pazocal{I} + r-(-\Delta)^{s})\ \psi^\epsilon|+\delta & \text{ in } \mathbb{R}^{n},
\end{cases}\]
where $\xi\in L^\infty((0,T]\times\mathbb{R}^n)$ is nonnegative, and $|\partial_t^h u^\epsilon(0,\cdot)-\partial_t u^\epsilon(0,\cdot)|\leq \delta$ for $h$ small for $\delta>0$. Hence, by the maximum principle, we have
\[
\|\partial^h_t u^\epsilon\|_{L^\infty((0,T]\times\mathbb{R}^n)}\leq e^{r T}(C\|\psi\|_{C^2(\mathbb{R}^n)}+\delta).
\]
Letting $\epsilon\longrightarrow 0^+$ and $h\longrightarrow 0^+$, we conclude the proof.
\end{proof} 

We notice that, since $u\in C^{1-0^+}((0,T];L^\infty(\mathbb{R}^n))\cap L^\infty((0,T];C^{2s-0^+}(\mathbb{R}^n))$ (see \Cref{mainthm1}), we have
\begin{equation}\label{gamma}
\pazocal{R}u\in L^\infty((0,T];C^{\gamma}(\mathbb{R}^n)) \quad \text{with} \quad \gamma\coloneqq s-\max\{\sigma,1/2\}.
\end{equation}
Here, we choose this exponent for simplicity, but the following general regularity holds:
\[
\pazocal{R}u\in L^\infty((0,T];C^{2\gamma-0^+}(\mathbb{R}^n)).
\]

\begin{lemma}
\label{comparison fractional and R} For a solution $u$ of \eqref{prin} and a fixed $t_0>0$, we have
\begin{equation}\label{equalset}
0\leq (-\Delta)^su(t_0, \cdot)-\pazocal{R} u(t_0, \cdot) < + \infty \ \text{ a.e. in } \ \{u(t_0, \cdot)=\psi\};
\end{equation}
\begin{equation}\label{diffset}
(-\Delta)^su(t_0, \cdot)-\pazocal{R} u(t_0, \cdot) \le 0 \ \text{ in } \ \{u(t_0, \cdot)>\psi\}.
\end{equation}
\end{lemma}

\begin{proof} 
Combining \Cref{cor:lipschitz} and \Cref{lem:semiconvexity}, we have that $\partial_tu\geq 0$ a.e. Also, $\partial_t u=0$ almost everywhere on the contact set $\{u=\psi\}$ so that
\[
\partial_t u+(-\Delta)^su-\pazocal{R}u=0 \;\text{ in }\; \{u>\psi\} \ \text{ and } \  \partial_t u=0\; \text{a.e. on}\;\{u=\psi\}.
\] This can be rewritten as
\begin{equation}\label{eqn:parabolic-vanish-open}
\partial_t u+(-\Delta)^su-\pazocal{R}u= \left( (-\Delta)^su-\pazocal{R}u \right)\chi_{\{u=\psi\}}.
\end{equation} This can be understood not only in the almost everywhere sense, but also in the distributional sense; incidentally, the right hand side is well defined by \Cref{uniformbound} and \Cref{cor:lipschitz} implies that $(-\Delta)^su-\pazocal{R}u$ is a bounded function. 

Notice that \eqref{eqn:parabolic-vanish-open} implies $\partial_t u+(-\Delta)^su-\pazocal{R}u$ is globally bounded and vanishes in the open set $\{u > \psi\}$, so that we can infer $u$ is smooth inside $\{u>\psi\}$. We are then allowed to write, for a fixed $t_{0} > 0$, 
\[
(-\Delta)^su(t_0, \cdot) - \pazocal{R}u(t_0, \cdot)=-\partial_t u(t_0, \cdot) \le 0  \ \text{ in } \ \{u(t_0, \cdot)>\psi\},
\] which is \eqref{diffset}. 

Next, since $\partial_t u=0$ a.e. on the contact set $\{u=\psi\}$, we have (see \Cref{uniformbound} and \Cref{cor:lipschitz}) that
\begin{equation}\label{staciobound}
0\leq (-\Delta)^su-\pazocal{R}u <\infty.
\end{equation} for almost every $(t,x)\,\in \{u = \psi\}$.

However, we need the same bound to hold for a.e. $x \in \mathbb{R}^n$, for every $t_0 \in (0,T]$. Note that Lipschitz continuity of $u$, see \Cref{cor:lipschitz}, implies that the map $t \longmapsto u(t, \cdot)\in L^2_{\operatorname{loc}}(\mathbb{R}^n)$ is uniformly continuous. In turn, by \eqref{equalset}, this implies weak continuity of the map
\begin{equation}\label{eqn:weakly-continuous}
t \longmapsto (-\Delta)^su(t, \cdot)-\pazocal{R}u(t, \cdot)\in L^2_{\operatorname{loc}}(\mathbb{R}^n).
\end{equation} Now, consider $\epsilon>0$, $A\subset \{u(t_0, \cdot)=\psi\}$ a bounded Borel set, multiply \eqref{staciobound} by $\chi_{[t_0-\epsilon,t_0]}\chi_A$, and integrate to obtain
\[\begin{split}
0 \le \int_{[t_0-\epsilon,t_0] \times A} \Big[ (-\Delta)^su-\pazocal{R}u \Big] \le C \, |A| \, \epsilon,
\end{split}\]
because, by \Cref{lem:semiconvexity}, $\{u(t,\cdot)=\psi\}$ is decreasing in time and thus so is $[t_0-\epsilon,t_0]\times A\subset \{u=\psi\}$. Since the map in \eqref{eqn:weakly-continuous} is weakly continuous, we obtain as $\epsilon \to 0$:
\[
0 \le \int_A \Big[ (-\Delta)^su (t_{0}, \cdot) -\pazocal{R}u(t_{0}, \cdot) \Big] \le C \, |A|,
\] for all bounded Borel set $A\subset \{u(t_0, \cdot)=\psi\}$. This concludes the proof.
\end{proof}

\section{H\"{o}lder-space decay of fractional Laplacian}

\noindent For a fixed $t > 0$, we assume, without loss of generality, that $0\in\partial\{u(t,\cdot)=\psi\}$ and consider the $L_a$-harmonic function $v: \mathbb{R}^n\times \mathbb{R}^+ \longrightarrow \mathbb{R}$ given by\footnote{By $\pazocal{R}u(t,0)$ we mean the evaluation of the function $\pazocal{R}u(t,\cdot)$ at the point $x = 0$.}
\[
v(x,y) \coloneqq u(t,x,y) + \frac{\pazocal{R}u(t,0)}{1-a}y^{1-a},
\]
where $u(t,x,y)$ denotes the harmonic extension of $u(t,x)$ to the upper half space, that is,
\[
\begin{cases}
L_au(t,x,y) \coloneqq \operatorname{div}_{x,y} \big[ y^a\nabla_{x,y}u(t,x,y) \big] =0 \ \text{ for } x \in \mathbb{R}^{n} \ \text{ and }  \ y>0, \\
u(t,x,0)=u(t,x).
\end{cases} 
\] See, for instance, Caffarelli-Silvestre \cite{cafsil} where the authors characterize the fractional Laplacian as\footnote{Actually, we have $\lim_{y\rightarrow 0^+}y^a\partial_yu(t,x,y)=-c_{n,a}(-\Delta)^su(t,x)$, and so we are taking the normalization constant as $c_{n,a}=1$ for simplicity.}
\begin{equation}\label{fractional}
\lim_{y\rightarrow 0^+} y^a u_y(t,x,y)=-(-\Delta)^su(t,x) \ \text{ with } \ a=1-2s.
\end{equation}

By \Cref{lem:semiconvexity}, we have
\[
u(t,x+h,0)+u(t,x-h,0)-2u(t,x,0)\leq-2C_0|h|^2 \ \text{ for every } \ h\in\mathbb{R}^n,
\] so that the maximum principle implies
\[
u(t,x+h,y)+u(t,x-h,y)-2u(t,x,y)\leq-2C_0|h|^2 \ \text{ for every } \ h\in\mathbb{R}^n \ \text{ and } \ y > 0.
\] This means that $u(t,x,y)$ is $C_0$-semiconvex with respect to $x$ for all $y\geq0$ and, in particular,
\[
\partial_y \big( y^au_y(t,x,y) \big) \le nC_0y^a.
\]
Now, consider the function
\[
\tilde{v}(x,y)\coloneqq v(x,y)-\psi(x)
\] and set $\Lambda \coloneqq \{\tilde{v}(x,0)=0\}=\{v(x,0)=\psi(x)\}$.

\begin{lemma}
\label{lem:b1b5} The following properties hold.
\begin{enumerate}[$(a)$]
	\item \label{b1} We have $\tilde{v} \ge 0$ in the set $\mathbb{R}^n\times\mathbb{R}^+ \setminus\Lambda\times\{0\}$;
	\vspace{2mm}
	\item \label{b2} The function $\tilde{v}$ is $2C_0$-semiconvex with respect to $x$ for all $y\geq0$ and
	\[
	\partial_y \big(y^a\tilde{v}_y(x,y)\big) \le 2nC_0y^a;
	\]
	\item \label{b3} For a.e. $x\in\Lambda$, 
	\[
	\lim_{y\rightarrow0^+}y^a\tilde{v}_y(x,y)\leq C_1|x|^\gamma
	\] and, for all $x\in\mathbb{R}^n\setminus\Lambda$, 
	\[
	\lim_{y\rightarrow0^+}y^a\tilde{v}_y(x,y)\geq -C_1|x|^\gamma;
	\]
	
	\item \label{b4} For all $x\in\Lambda$,
	\[
	\tilde{v}(x,y)-\tilde{v}(x,0) \le \frac{nC_0}{1+a}y^2+\frac{C|x|^\gamma}{1-a} y^{1-a};
	\]
\end{enumerate}
\end{lemma}

\begin{proof}  The first item only restates that $u(t,x) \ge \psi(x)$. Next, \eqref{b2} follows from the semiconvexity of $v$ and $\psi$: we obtain that $\tilde{v}$ is $2C_0$-semiconvex with respect to $x$ for all $y \ge 0$, which implies
\[
\partial_y(y^a\tilde{v}_y(x,y))\leq2nC_0y^a.
\] In order to show \eqref{b3}, we first use \eqref{equalset} and \eqref{gamma} to conclude that, for a.e. $x\in\Lambda$,
\[
\lim_{y\rightarrow0^+}y^a\tilde{v}_y(x,y)=-(-\Delta)^su(t,x)+\pazocal{R}u(t,0)\leq |\pazocal{R}u(t,x)-\pazocal{R}u(t,0)|\leq C_1|x|^\gamma.
\] Then, we use \eqref{diffset} (and again \eqref{gamma}) to obtain that, for every $x\in\mathbb{R}^n\setminus\Lambda$,
\[
\lim_{y\rightarrow0^+}y^a\tilde{v}_y(x,y)=-(-\Delta)^su(t,x)+\pazocal{R}u(t,0)\geq \pazocal{R}u(t,0)-\pazocal{R}u(t,x)\geq-C_1|x|^\gamma.
\] Now we prove \eqref{b4}. For a.e. $x\in\Lambda$, we have
\[
\begin{split}
\tilde{v}(x,y)-\tilde{v}(x,0) & = \int_0^y\frac{s^a\tilde{v}_y(x,s)}{s^a} \, \dd s 
= \int_0^y \frac{1}{s^a} \bigg( \int_0^s\partial_y(\tau^a\tilde{v}_y(x,\tau)) \, \dd \tau + \lim_{z\rightarrow0^+} z^a\tilde{v}_y(x,z) \bigg) \, \dd s \\ 
     & \le \int_0^y\frac{1}{s^a}\left(\frac{2nC_0s^{a+1}}{a+1}+C|x|^\gamma\right) \, \dd s = \frac{nC_0}{1+a}y^2+\frac{C|x|^\gamma}{1-a} y^{1-a},
\end{split}
\] where the inequality relies in \eqref{b2} and \eqref{b3} above. Moreover, by continuity, the estimate holds for every $x\in\Lambda$.
\end{proof}

Now, let us analyze a first decay property of $y^a\tilde{v}_y$.

\begin{proposition}\label{prop:inf1}
There exists $c  > 0$ and $\mu \in (0,1)$ for which
\begin{equation}\label{prop1}
\inf_{\Gamma_{4^{-k}}}  y^a\tilde{v}_y (x,y) \ge -c \mu^k,
\end{equation} where $\Gamma_r \coloneqq B_r\times[0,\eta r]$ and $\eta \coloneqq \sqrt{\frac{1+a}{2n}}$.
\end{proposition}

\begin{proof}
The result follows by induction. To obtain the case $k=0$, we note
\[
\begin{split}
L_{-a}(y^a\tilde{v}_y) & = \operatorname{div}_{x,y} \left( y^{-a} \nabla_{x,y} (y^a\tilde{v}_y)\right) = \operatorname{div}_{x,y}\left(y^{-a}\nabla_{x,y} ( y^au_y ) \right) \\
& = \Delta_x (u_y) + \partial_y (y^{-a} \partial_y (y^au_y)) 
= \partial_y\left(\Delta_xu + y^{-a}\partial_y(y^au_y)\right) \\ &
= \partial_y(y^{-a}L_au(t,x,y))=0.
\end{split}\]
Then, since 
\[
\lim_{y\rightarrow0^+} y^a \tilde{v}_y(x,y) = - (-\Delta)^s u(t,x) + \pazocal{R}u(t,0)
\] is a bounded function, we obtain that $y^a\tilde{v}_y(x,y)$ remains bounded, for $y>0$, by the maximum principle. This is enough for the case $k = 0$.

Now, assume that \eqref{prop1} holds for some $k \in \mathbb{N}$, where $c $ and $\mu$ are to be chosen later. Set 
\[\tilde{V}(x,y) \coloneqq \frac{4^{2sk}}{c  \mu^{k}} \, \tilde{v} \left(\frac{x}{4^{k}},\frac{y}{4^{k}}\right).
\] The induction hypothesis (recall $a=1-2s$ and $v_{y} = \tilde{v}_{y}$) reads
\begin{equation}\label{eqn:prop-inf1-ii}
\begin{split}
\inf_{\Gamma_1} (y^a\tilde{V}_y) = \inf_{\Gamma_1} (y^a\bar{V}_y) = \frac{1}{c  \mu^{k}} \, \inf_{\Gamma_1} \left[ \frac{y^{a}}{4^{ka}} \, \tilde{v}_y \left( \frac{x}{4^{k}}, \frac{y}{4^{k}} \right) \right] = \frac{1}{c  \mu^{k}} \, \inf_{\Gamma_{4^{-k}}} y^{a} \tilde{v}_y (x, y) \geq -1.
\end{split}
\end{equation} So, in this renormalized notation, it is enough to show that
\[
\inf_{\Gamma_{1/4}} y^{a} \tilde{V} (x,y) > - \mu.
\] In order to do that, consider the auxiliary function
\begin{equation}\label{vbar}
\bar{V}(x,y) \coloneqq \frac{4^{2sk}}{c  \mu^{k}} \,  \bar{v}\left(\frac{x}{4^{k}},\frac{y}{4^{k}}\right) \quad \text{where} \quad  \bar{v}(x,y) \coloneqq v(x,y) - \Big( \psi(0) + \nabla \psi(0) \cdot x \Big).
\end{equation}
Both $\bar{v}$ and $\bar{V}$ are $L_a$-harmonic functions. Also, by using the $C_0$-semiconvexity of $v$, we obtain 
\begin{equation}\label{eqn:prop-inf1-i}
\begin{split}
|\tilde{V}(x,y)-\bar{V}(x,y)| & \leq \frac{4^{2sk}}{c  \mu^{k}} \, \left| \tilde{v}\left(\frac{x}{4^{k}},\frac{y}{4^{k}}\right)-\bar{v}\left(\frac{x}{4^{k}},\frac{y}{4^{k}}\right)\right| \\ & \le \frac{4^{2sk}}{c  \mu^{k}} \, \big|\psi(0) - \psi(4^{-k}x) + \nabla\psi(0) \cdot (4^{-k} x) \big| \\ & \le \frac{C_0 4^{2(s-1)k}}{2c  \mu^{k}} |x|^2.
\end{split}
\end{equation} Moreover, \Cref{lem:b1b5}\eqref{b2} yields
\begin{equation}\label{eqn:prop-inf1-iii}
\partial_y(y^a\tilde{V}_y)=\partial_y(y^a\bar{V}_y)\leq\frac{2nC_0}{c  4^{2(1-s)k}\mu^{k}} \, y^a.
\end{equation} Furthermore, both $\tilde{V}$ and $\bar{V}$ are semiconvex in the set $\Gamma_1$ with constant $\frac{2C_0}{c  4^{2(1-s)k}\mu^{k}}$.

Let us fix $L \gg C_0$ yet to be chosen. As can be checked below, we can assume this constant depends only $n$, $a$, and $C_{0}$. Set
\[
\bar{W}(x,y) \coloneqq \bar{V}(x,y) + \frac{\|\pazocal{R}u(t, \cdot)\|_{C^\gamma(B_1)}}{c  4^{\gamma k} \mu^{k} (1-a)} \, y^{1-a} - \frac{L}{c  4^{2(1-s)k} \mu^{k}} \left( |x|^2 - \frac{n}{1+a} y^2 \right).
\] We have the following properties:
\begin{enumerate}[$(i)$]
	\item By a straightforward computation, $\bar{W}$ is an $L_a$-harmonic function.% in $\Gamma_{1/8}$.
	
	\item The semiconvexity of $\psi$ implies that, for every $x\in\Lambda\setminus\{0\}$,
	\[
	\begin{split}
	\bar{W}(x,0) & = \frac{4^{2sk}}{K_{1} \mu^{k}} \Bigg[ v \left( \frac{x}{4^{k}}, 0 \right) - \psi(0) - \nabla \psi(0) \cdot \frac{x}{4^{k}} - L \left| \frac{x}{4^{k}} \right|^{2} \Bigg] \\ 
	&  \le \frac{4^{2sk}}{K_{1} \mu^{k}} \Bigg[ v \left( \frac{x}{4^{k}}, 0 \right) - \psi\left( \frac{x}{4^{k}} \right) - (L - C_{0}) \left| \frac{x}{4^{k}} \right|^{2} \Bigg] < 0.
	\end{split}
	\]
	
	\item By the continuity of $v$,
	\[
	\lim_{(x,y)\rightarrow(0,0)} \bar{W}(x,y) = \frac{4^{2sk}}{c  \mu^{k}} \Big[ v(0,0) - \psi(0) \Big] = 0.
	\]
	\item\label{propiv} By \Cref{lem:b1b5}\eqref{b3}, we have, for $|x|< B_{1/8}\setminus\Lambda$,
	\[
	\lim_{y\rightarrow0^+}y^a\bar{W}_y(x,y)>0.
	\] Indeed, we have
		\[
		\begin{split}
		y^{a} \bar{W}_{y} & = y^{a} \bar{V}_{y} + \frac{\|\pazocal{R}u(t, \cdot)\|_{C^\gamma(B_1)}}{c  4^{\gamma k} \mu^{k}} + \frac{2nL}{c  4^{2(1-s)k} \mu^{k} (1+a)} \, y^{1+a} \\
		& = \frac{1}{c  \mu^{k}} \left[ \left( \frac{y}{4^{k}} \right)^{a} \tilde{v}_y \left( \frac{x}{4^{k}}, \frac{y}{4^{k}} \right) + \frac{\|\pazocal{R}u(t, \cdot)\|_{C^\gamma(B_1)}}{4^{\gamma k}} + \frac{2nL}{4^{2(1-s)k} (1+a)} \, y^{1+a} \right].
		\end{split}
		\] Let $y \to 0^{+}$ and recall that $C_{1}$ is a H\"older constant for $\pazocal{R}u(t, \cdot)$ to infer that
		\[
		\lim_{y \to 0^{+}} y^{a} \bar{W}_{y} \ge \frac{1}{c  4^{\gamma k} \mu^{k}} \left[ - C_{1} |x|^{\gamma}  + \|\pazocal{R}u(t, \cdot)\|_{C^\gamma(B_1)} \right] > 0.
		\] 
\end{enumerate}
In particular, \eqref{propiv} implies that for a fixed $x \in B_{1/8} \setminus \Lambda$, $\bar{W}(x,y)>\bar{W}(x,0)$ for all $(x,y)\in (B_{1/8} \setminus \Lambda)\times(0,\delta)$, once $\delta > 0$ is small enough so that $\bar{W}_{y}(x,y)>0$ for all $y\in[0,\delta)$.

These properties and Hopf's Lemma (see, for instance, \cite[Theorem 3.5]{GiTr-01}) imply that the maximum of $\bar{W}$ is non-negative and attained on $\partial\Gamma_{1/8}\setminus \{y=0\}$. Hence, this maximum is achieved either at a point on the top $\partial \Gamma_{1/8} \cap \{y = \eta/8\}$ of the cylinder or at a point on the side $\partial B_{1/8} \times (0, \eta/8)$. In what follows, we analyze each case separately.

If the maximum is attained on $\partial \Gamma_{1/8}\cap \{y=\eta /8\}$, there exists $x_0 \in B_{1/8}$ for which $\bar{W}(x_{0}, \eta/8) \ge 0$. Thus, we have
\[
\bar{V}(x_0,\eta/8) + A \, \frac{\|\pazocal{R}u(t, \cdot)\|_{C^\gamma(B_1)}}{c  4^{\gamma k} \mu^{k}} \ge - B \, \frac{L}{c  4^{2(1-s)k} \mu^{k}},
\] where $A \coloneqq \frac{\eta^{1-a}}{(1-a)8^{1-a}}$ and $B \coloneqq \frac{n \eta^{2}}{64(a+1)}$. Since $\eta$ depends only of $n$ and $a$, so do the positive constants $A$ and $B$. By the semiconvexity of $\bar{V}$, see \eqref{eqn:prop-inf1-iii}, we can write
\[
\bar{V}(x,\eta/8) \ge \bar{V}(x_{0},\eta/8) + \langle \nabla_{x} \bar{V}(x_{0},\eta/8), x - x_{0} \rangle - \frac{2C_0}{c  4^{2(1-s)k}\mu^{k}} \, |x - x_{0}|^{2},
\] so that, in the half-ball
\[
\operatorname{HB}_{1/2}(x_{0}, \eta/8) \coloneqq \{ z \in B_{1/2} (x_{0}) ; ~ \langle \nabla_{x} \bar{V}(x_{0},\eta/8), z - x_{0} \rangle \ge 0 \},
\] there holds
\begin{equation}\label{eqn:prop-inf1-semiconv}
\bar{V}(x,\eta/8) + A \, \frac{\|\pazocal{R}u(t, \cdot)\|_{C^\gamma(B_1)}}{c  4^{\gamma k} \mu^{k}} \ge  - \frac{BL + C_{0}}{c  4^{2(1-s)k} \mu^{k}} \ge - \frac{2 B L}{c  4^{2(1-s)k} \mu^{k}}.
\end{equation} In the last inequality, we use the fact that $L$ is choosen much larger than $C_{0}$.
%%%%%%
%%%%% obtained in \eqref{eqn:prop-inf1-iii}, property \Cref{lem:b1b5}\eqref{b5}, and $C_0\ll L$, we have that, for all $x\in \operatorname{HB}_{1/2}(x_0,\eta/8)$, \textcolor{red}{fazer mais passo a passo}
%%%%%\[
%%%%%\bar{V}(x,\eta/8) + A \, \frac{\|\pazocal{R}u(t, \cdot)\|_{C^\gamma(B_1)}}{c  4^{\gamma k} \mu^{k}} \ge - 2 B \, \frac{L}{c  4^{2(1-s)k} \mu^{k}}.
%%%%%\] 
%%%%%
%
Now, recall $\bar{V}_y=\tilde{V}_y$; hence, \Cref{lem:b1b5}\eqref{b3} gives 
\begin{equation}\label{lim}
\begin{split}
\displaystyle \lim_{y\rightarrow 0^+}y^a \bar{V}_y(x,y)\leq \frac{C_1}{c  4^{\gamma k}\mu^{k}} \, |x|^\gamma& \ \text{ if } \ \tilde{V}(x,0)=0 \ \text{ and} \\ 
\displaystyle \lim_{y\rightarrow 0^+}y^a \bar{V}_y(x,y)\geq - \frac{C_1}{c  4^{\gamma k}\mu^{k}} \, |x|^\gamma& \ \text{ if } \ \tilde{V}(x,0)>0.
\end{split}
\end{equation} Integrate \eqref{eqn:prop-inf1-iii} with respect to $y$ in the interval $[0,y]$, with $y<\eta/8$ to obtain
\[
\lim_{y \to 0^{+}} y^a\bar{V}_y(x,y) + \frac{2nC_0 \eta^{a + 1}}{c  4^{2(1-s)k}\mu^{k} (a + 1) 8^{a + 1}} \ge \frac{\eta^{a} \bar{V}_{y} (x,y)}{8^{a}}.
\] Integrating the inequality above with respect to $y$ in the interval $[0,\eta/8]$ combined with \eqref{eqn:prop-inf1-semiconv} and \eqref{lim}  yield, for all $x\in \operatorname{HB}_{1/2}(x_0,\eta/8)$, 
\[
\lim_{y\rightarrow 0^+}y^a\bar{V}_y(x,y) + A \frac{\|\pazocal{R}u(t, \cdot)\|_{C^\gamma(B_1)}}{c  4^{\gamma k} \mu^{k}} \ge  - B' \frac{L}{c (4^{2(1-s)}\mu)^{k}}
\] where $A' = \frac{\eta^{a-1}A}{8^{a-1}}$ and $B' = \frac{2B \eta^{a-1}}{8^{a-1}} \! + \!\frac{2n \eta^{a+1}}{(a+1) 8^{a+1}}\!+\! \frac{1}{4}$ are positives constants that depend only on $a$ and $n$. This is again possible because of the choice $L \gg C_{0}$.

%%%%%-------------------------------------------------------

On the other hand, suppose the non-negative maximum of $\bar{W}$ is attained on a point $(x_0,y_0) \in \partial B_{1/8} \times (0, \eta/8)$. The definition of $\eta$ implies $0 \le y_{0}^{2} \le \frac{1 + a}{2n 8^{2}} = \frac{1 + a}{2n} |x_{0}|^{2}$. Thus, since $\bar{W} (x_{0}, y_{0}) \ge 0$,
\[
\bar{V}(x_0,y_0) + D' \, \frac{\|\pazocal{R}u(t, \cdot)\|_{C^\gamma(B_1)}}{c  4^{\gamma k} \mu^{k}} \ge \frac{L}{2^{7} c  4^{2(1-s)k} \mu^{k}},
\] where $D' = \frac{\eta^{2(1-a)}}{(1-a)8^{2(1-a)}}$. We can repeat the argument of the previous case to obtain that 
\[
\lim_{y\rightarrow 0^+}y^a\bar{V}_y(x,y) + D'' \, \frac{\|\pazocal{R}u(t, \cdot)\|_{C^\gamma(B_1)}}{c  4^{\gamma k} \mu^{k}} \ge -B''\frac{L}{c (4^{2(1-s)}\mu)^{k}},
\] for all $x\in \operatorname{HB}_{1/2}(x_0,y_0)$, where $D'' = D'+\frac{\eta^{a-1}}{8^{a-1}}$, and $B'' = \frac{2B \eta^{a-1}}{8^{a-1}} \!+\! \frac{1}{4}$. 

In any case, there exist $C > 0$, $D > 0$,  $\bar{y}\in [0,\eta/8]$, and $\bar{x} \in B_{1/8}$ such that, for all $x\in \operatorname{HB}_{1/2}(\bar{x},\bar{y})$,
\[
\lim_{y\rightarrow 0^+} y^a \bar{V}_y (x,y) \ge - D \, \frac{\|\pazocal{R}u(t, \cdot)\|_{C^\gamma(B_1)}}{c  4^{\gamma k} \mu^{k}} - \frac{ C }{c  4^{2(1-s)k} \mu^{k}}.
\] We observe the constants above depend only on $n$, $a$, and $C_{0}$. The choices 
\[
\max\{4^{-\gamma},4^{-2 + 2s}\} \le \mu < 1 \ \text{ and } \ c  > 2\big( C + D \|\pazocal{R}u(t, \cdot)\|_{C^\gamma(B_1)} \big)
\] then provides us with
\begin{equation}\label{geq12}
\lim_{y\rightarrow 0^+}y^a\bar{V}_y(x,y)>-\frac{1}{2}.
\end{equation}
%%%%%Also, by , we have \textcolor{red}{talvez tirar, só repete (3.3)}
%%%%%\begin{equation}\label{geq1}
%%%%%y^a\bar{V}_y(x,y)\geq -1\;\;\; \text{in}\;\;\; \Gamma_1.
%%%%%\end{equation}
As in case $k=1$, we have that $y^a\bar{V}_y(x,y)$ solves $L_{-a}(y^a\bar{V}_y(x,y))=0$ in $\mathbb{R}^n\times\mathbb{R}^+$. From this, we now show that \eqref{geq12} and \eqref{eqn:prop-inf1-ii} imply that there exists $\theta<1$ such that, for every $x \in B_{1/4}$,
\begin{equation}\label{boundb14}
\left(\frac{\eta}{4}\right)^a\bar{V}_y(x,\eta/4)\geq -\theta.
\end{equation} Indeed, by the minimum principle, we have
\begin{equation}\label{prooftheta}
\inf_{x\in B_{5/8}}\lim_{y\rightarrow 0^+}y^a\bar{V}_y(x,y)\leq \inf_{(x,y)\in\Gamma_{5/8}}y^a\bar{V}_y(x,y)\leq \left(\frac{\eta}{4}\right)^a\bar{V}_y(x,\eta/4)
\end{equation}
for all $x\in B_{1/4}$. Then, \eqref{eqn:prop-inf1-ii} and Harnack's inequality yield
\[
1+\sup_{x\in B_{5/8}}\lim_{y\rightarrow 0^+}y^a\bar{V}_y(x,y)\leq C\left(\inf_{x\in B_{5/8}}\lim_{y\rightarrow 0^+}y^a\bar{V}_y(x,y)+1\right),
\]
for some constant $C>0$ depending only on $s$ and $n$.
 Since $ \operatorname{HB}_{1/2}(\bar{x},\bar{y})\subset B_{5/8}$, by \eqref{geq12}, we have
\[
\inf_{x\in B_{5/8}}\lim_{y\rightarrow 0^+}y^a\bar{V}_y(x,y)\geq \frac{1}{C}\sup_{x\in \operatorname{HB}_{1/2}(\bar{x},\bar{y})}\lim_{y\rightarrow 0^+}y^a\bar{V}_y(x,y)-1+\frac{1}{C}\geq \frac{1}{2C}-1\eqqcolon -\theta.
\] By the above and \eqref{prooftheta}, we conclude \eqref{boundb14}.
 
 Next, integrate \eqref{eqn:prop-inf1-iii} with respect to $y$ in the interval $[y, \eta/4]$ to obtain
\[
\left( \frac{\eta}{4} \right)^{a} \bar{V}_y (x, \eta/4) - y^a\bar{V}_y (x,y) \le \frac{2nC_0}{c  4^{2(1-s)k}\mu^{k} (a + 1)} \, \Big[ (\eta/4)^{a + 1} - y^{a + 1} \Big] \le \frac{\hat{C}}{c },
\] where $\hat{C} = \frac{2nC_0\eta^{a + 1}}{(a + 1)4^{a + 1}}$ is a positive constant that depends only on $n$, $a$, and $C_{0}$. We thus have
\[
y^a\bar{V}_y(x,y) \ge - \theta - \frac{\hat{C}}{c }.
\] First enlarge, if necessary, $c $ so that $\theta + \hat{C}/K_{1} < 1$; then, enlarge $\mu$ (if necessary) so that $\theta + \hat{C}/K_{1} < \mu < 1$. Therefore,
\[
y^a\tilde{V}_y(x,y) = y^a\bar{V}_y(x,y) > -\mu, 
\] for every $x \in B_{1/4}$ and every $y \in [0, \eta/4]$, which is what we wanted.
\end{proof}

Once \Cref{prop:inf1} is established, we show in a standard manner (see, for instance, \cite[Lemma 4.4]{CaFi-08}) how a bound from below of the form $\inf_{\Gamma_r}y^a\tilde{v}_y \ge - C r^\alpha$ provides control of the $L^{\infty}$-norm of $\tilde{v}$ in a smaller cylinder.

\begin{lemma}\label{lem:inf-to-L_infty}
For $C>0$, $\alpha\in(0,1)$, and $r\in(0,1]$ such that $\inf_{\Gamma_r}y^a\tilde{v}_y \ge - C r^\alpha$, there exists $M > 0$ for which
\[
\sup_{\Gamma_{r/8}}|\tilde{v}|\leq Mr^{\alpha+2s}.
\] Moreover, the constant $M$ is independent of $r$ and depends only on $C,\alpha,a$, and $C_0$.
\end{lemma}

\begin{proof}
We consider only the case where $r > 0$ is small, for $\tilde{v}$ is globally bounded. By \Cref{lem:b1b5}\eqref{b1} and by our assumption, we have, for every $(x,y)\in \Gamma_r$,
\[
\tilde{v}(x,y) \ge \tilde{v}(x,0) - C r^\alpha \int_0^y \tau^{-a} \, \dd \tau \ge - \, \frac{C \eta^{1-a}}{1-a} \, r^{\alpha+2s}.
\] This provides a lower bound on $\tilde{v}$.

Let us assume, by contradiction, that the upper bound does not hold, that is, for any $M > 0$, there exists $(x_{0},y_{0})\in\Gamma_{r/8}$ such that $\tilde{v}(x_{0},y_{0})\geq Mr^{\alpha+2s}$. Our assumption, by integration, yields
\[
\tilde{v}(x_{0}, \eta r/2) \ge \tilde{v}(x_{0},y_{0}) - \frac{C \eta^{2s}}{(1-a) 2^{2s}} \, r^{\alpha + 2s} + \frac{C y_{0}^{2s}}{1-a} \, r^\alpha \ge \left( M - \frac{C \eta^{2s}}{(1-a)2^{2s}} \right) r^{\alpha + 2s}.
\] In particular, for sufficiently large $M > 0$, namely $M \ge \frac{4 C \eta^{2s}}{3 (1-a) 2^{2s}}$, we can write
\[
\tilde{v}(x_{0},\eta r/2)\geq \frac{M}{4} \, r^{\alpha+2s}.
\] Next, denote $\bar{v}$ as in \eqref{vbar} and observe that the semiconvexity of $\psi$ implies $|\bar{v} - \tilde{v}| \le C_0r^2$ in $\Gamma_r$. Then, the lower bound above gives
\[
\bar{v}(x,y) + \frac{C r^{\alpha+2s}}{1-a}+C_0r^2 \ge 0 \ \text{ for every } \ (x,y) \in \Gamma_r.
\] Now, $B_{\eta r/2} (x_{0}, \eta r/2) \subset \Gamma_r$ and $(0,\eta r/2) \in B_{\eta r/4} (x_{0}, \eta r/2)$, so that Harnack inequality, applied in $B_{\eta r/2} (x_{0}, \eta r/2)$, gives 
\[
\frac{M}{4} r^{\alpha+2s} \le \sup_{B_{\eta r/4}} \left[ \bar{v} + \frac{C r^{\alpha+2s}}{1-a} + C_0 r^2 \right] \le c \left(\bar{v}(0,\eta r/2) + \frac{C r^{\alpha+2s}}{1-a}+C_0r^2\right).
\] Hence, there exists $c_0>0$ such that
\[
\tilde{v}(0,\eta r/2) + C_{0} r^{2} \ge \bar{v}(0,\eta r/2) \ge c_0 M r^{\alpha+2s} - \frac{C r^{\alpha+2s}}{1-a} - C_0 r^2.
\] Recall $0\in \Lambda$; then, by \Cref{lem:b1b5}\eqref{b4},
\[
\begin{split}
0 = \tilde{v}(0,0) & \ge \tilde{v} (0, \eta r /2)  - \frac{nC_0 \eta ^{2}}{4(1+a)} r^2 \\ & \ge c_0 M r^{\alpha+2s} - \frac{C r^{\alpha+2s}}{1-a} - \frac{nC_0 \eta ^{2}}{4(1+a)} r^2 - 2 C_0 r^2.
\end{split}
\] In particular, we have a bound for $M$:
\[
M \le \frac{1}{c_0r^{\alpha+2s}} \left( \frac{C r^{\alpha+2s}}{1-a} + \frac{nC_0 \eta ^{2}}{4(1+a)} r^2 + 2C_0 r^2 \right).
\] This is in contradiction to our assumption because the constant $M > 0$ should be arbitrary. 
\end{proof}

We are now in a position to prove a first regularity estimate at a free boundary point.

\begin{theorem}
\label{holderthm} Let $u$ be a solution of \eqref{prin}, and $\psi$, $b$, $r$, and $\pazocal{I}$ as in \eqref{psi}, \eqref{defb}, \eqref{defr}, and \eqref{K}, respectively. Then, there exist $\bar{C}>0$ and $\alpha\in(0,\gamma)$ such that, for every $r \in (0,1)$ and every $x_{0} \in \partial\{u(t, \cdot)=\psi\}$,
\begin{equation}\label{alpha2s}
\sup_{B_r(x_{0})} |u(t, \cdot) - \psi| \le \bar{C} \, r^{\alpha+2s} \ \text{ and }
\end{equation}
\begin{equation}\label{alpha}
\sup_{B_r(x_{0})} \, \Big| \, \big[(-\Delta)^su(t, \cdot) - \pazocal{R}u(t, \cdot) \big] \, \chi_{\{u(t, \cdot) = \psi\}} \, \Big| \le \bar{C} \, r^\alpha.
\end{equation}
\end{theorem}

\begin{proof}  
	The estimate in \eqref{alpha2s} is a direct consequence of \Cref{lem:inf-to-L_infty}. In order to prove \eqref{alpha}, we assume, as before, $x_{0} = 0$.  Recall that, by the definition of $\tilde{v}$,
\[
(-\Delta)^s u(t,x') - \pazocal{R}u(t,0) = - \lim_{y\rightarrow 0^+} y^a \tilde{v}_y (x',y) = - \lim_{y\rightarrow 0^+} y^a v_y (x',y)
\] and so
\[
(-\Delta)^s u(t,x') - \pazocal{R}u(t,x') = - \lim_{y\rightarrow 0^+} y^a v_y(x',y) + \pazocal{R}u(t,0) - \pazocal{R} u(t,x').
\]  By \eqref{equalset} and \eqref{gamma} we have that
\[
\sup_{B_r} \, \Big| \, \big[(-\Delta)^su(t, \cdot) - \pazocal{R}u(t, \cdot) \big] \, \chi_{\{u(t, \cdot) = \psi\}}\Big|\leq -\inf_{B_r}\lim_{y\rightarrow 0^+} y^a v_y(\cdot,y) + C_{1}r^\gamma.
\]
Now, if $1/4<r<1$, then by \Cref{prop:inf1} for $k=0$ we have that
\[
-\inf_{B_r}\lim_{y\rightarrow 0^+} y^a v_y(x',y)\leq c \leq 4c  r;
\]
on the other hand, if $r\leq 1/4$, by taking $\beta$ such that $\beta\leq \log_4\mu^{-1}$ combined with \Cref{prop:inf1}, we obtain
\[
-\inf_{B_r}\lim_{y\rightarrow 0^+} y^a v_y(x',y)\leq -\inf_{B_{1/4}}\lim_{y\rightarrow 0^+} y^a v_y(x',y)\leq c \mu \leq c \mu^k\leq c  4^{-k\beta}.
\]
Hence, choosing $k$ large enough so that $4^{-k}<r$ gives \eqref{alpha} for $\alpha=\min\{\beta,\gamma\}$.
\end{proof}

\begin{corollary}\label{cor:holderthm}
In the same setting of \Cref{holderthm}, there exist $\bar{C}'>0$ and $\alpha \in (0,\gamma)$ such that
\[ 
\Big\| \, \big[(-\Delta)^su(t, \cdot) - \pazocal{R}u(t, \cdot) \big] \, \chi_{\{u(t, \cdot) = \psi\}} \, \Big\|_{C^\alpha(\mathbb{R}^n)}\leq \bar{C}',
\]
that is,
\[
\big[(-\Delta)^su(t, \cdot) - \pazocal{R}u(t, \cdot) \big] \, \chi_{\{u(t, \cdot) = \psi\}} \in C^\alpha(\mathbb{R}^n)\cap L^\infty(\mathbb{R}^n).
\]
\end{corollary}

\begin{proof} Let $\alpha$ obtained in \Cref{holderthm}. Over the set $\Lambda = \{u(t, \cdot) = \psi\}$, recall that the function $(-\Delta)^s u(t, \cdot) - \pazocal{R}u(t, \cdot)$ is bounded, by \eqref{equalset}. It is then enough to show that, for $|x_1-x_2| \le 1/4$ with $x_1,x_2 \in \Lambda$, 
\[
\Big|(-\Delta)^su(t,x_1)-(-\Delta)^su(t,x_2)-\pazocal{R}u(t,x_1)+\pazocal{R}u(t,x_2) \Big| \le C |x_{1} - x_{2}|^{\alpha}.
\] Given $x \in \Lambda$, let $d(x, \partial \Lambda)$ denote the distance from $x$ to $\partial \Lambda$. We then analyze two possible situations.
\begin{itemize}
\item Suppose first that 
\[
|x_1-x_2| \le \frac{1}{4} \max \big\{ d(x_{1}, \partial \Lambda), d(x_{2}, \partial \Lambda) \big\}.
\] %%%% which implies $u(t, \cdot) = \psi$ in $S\coloneqq B_{4|x_1-x_2|}(x_1)\cap B_{4|x_1-x_2|}(x_2)$.
By \Cref{holderthm}, we have, for any $r \in (0,1)$,
\[
\sup_{B_r(x_i)}|u(t, \cdot)-\psi| \le \bar{C} r^{\alpha+2s}.
\] In particular, $u(t, \cdot) = \psi$ in the set $S\coloneqq B_{4|x_1-x_2|}(x_1)\cap B_{4|x_1-x_2|}(x_2)$. Also, we trivially have 
\[
|u(t, \cdot) - \psi| \le M  \coloneqq \| u(t, \cdot) - \psi \|_{L^\infty(\mathbb{R}^n)} 
\] outside the set $B_1(x_1) \supset B_{1/2}(x_2)$, and then
\[\begin{split}
|(-\Delta)^s f(x_1)-(-\Delta)^sf(x_2)-\pazocal{R}f(x_1)
+\pazocal{R}f(x_2)|\leq C_1|x_1-x_2|^\gamma\\+\int_{\mathbb{R}^n\setminus S}|f(x')|\left|\frac{1}{|x'-x_1|^{n+2s}}-\frac{1}{|x'-x_2|^{n+2s}}\right| \, \dd x'\\
\leq C_1|x_1-x_2|^\gamma+C\left[\bar{C}\int_{|x_1-x_2|}^1 \tau^{\alpha'-2} \, \dd \tau +M\right]|x_1-x_2|\leq C|x_1-x_2|^\alpha,
\end{split}\]
where $f\coloneqq u(t,\cdot)-\psi$. Because $\|(-\Delta)^s\psi\|_{C_x^{1-s}(\mathbb{R}^n)}$ is bounded, this gives the result.

\item If, on the other hand,
\[
|x_1-x_2| \ge \frac{1}{4} \max \big\{ d(x_{1}, \partial \Lambda), d(x_{2}, \partial \Lambda) \big\},
\] we take $\bar{x}_1, \bar{x}_2 \in \partial\Lambda$  for which $|x_1-\bar{x_1}| = d(x_{1}, \partial \Lambda)$ and $|x_2-\bar{x_2}| = d(x_{2}, \partial \Lambda)$. Therefore, by \Cref{holderthm}, we have
\[\begin{split}
|(-\Delta)^s f(x_1)-(-\Delta)^sf(x_2)-\pazocal{R}f(x_1)
+\pazocal{R}f(x_2)|\leq C_1|x_1-x_2|^\gamma\\
+\sup_{B_{4|x_1-x_2|}(\bar{x}_1)}|(-\Delta)^s f|+\sup_{B_{4|x_1-x_2|}(\bar{x}_2)}|(-\Delta)^s f|\leq \bar{C}'|x_1-x_2|^\alpha.
\end{split}\]
\end{itemize}
\end{proof} 

\section{Monotonicity formula and optimal regularity in space}

We recall a regularity property provided by the fractional heat operator (see, for instance, \cite[Appendix A]{CaFi-08}); namely, if $v$ satisfies 
\[
\partial_t v + (-\Delta)^s v = f
\] with $f\in L^\infty((0,T];C^{\beta}(\mathbb{R}^n))$ for some $\beta\in(0,1)$, then
\begin{equation}\label{regularityheateq}
\|\partial_t v\|_{L^\infty((0,T];C^{\beta-0^+}(\mathbb{R}^n))}+\|(-\Delta)^sv\|_{L^\infty((0,T];C^{\beta-0^+}(\mathbb{R}^n))} \le C \left(1+ \|f\|_{L^\infty((0,T];C^{\beta}(\mathbb{R}^n))}\right).
\end{equation} Incidentally, we have shown in \Cref{cor:holderthm} and \eqref{gamma} that
\[
\partial_t u + (-\Delta)^s u = \big[(-\Delta)^s u - \pazocal{R}u \big] \chi_{\{u=\psi\}} + \pazocal{R} u \in L^\infty\big((0,T];C^\alpha(\mathbb{R}^n)\big)
\] so that \eqref{regularityheateq} holds for our solution $u$. Hence, $(-\Delta)^su\in C^{\alpha-0^+}$, and since $u$ is bounded \Cref{uniformbound}, by \cite[Proposition 2.1.8]{silvestre}, we have
\[
u\in L^\infty \big( (0,T]; C^{2s+\alpha-0^+}(\mathbb{R}^n) \big).
\] 
Moreover, the lower order term has the regularity
\begin{equation}\label{initialregr}
\pazocal{R}u\in L^\infty \big( (0,T]; C^{\alpha+\gamma}(\mathbb{R}^n) \big).
\end{equation}

Next, we consider $0 \in \partial\{u(t, \cdot) = \psi\}$. Moreover, let $w:\mathbb{R}^n\times \mathbb{R}^+ \rightarrow \mathbb{R}$ be the function which solves, with fixed $t > 0$, 
\begin{equation}\label{eqn:w}
\begin{cases}
L_{-a} w = 0 & \text{ in } \mathbb{R}^{n}\times \mathbb{R}^+,\\
w(x,0) = \big[(-\Delta)^su (t,x) - \pazocal{R} u (t,0)\big] \chi_{\{u(t, \cdot) = \psi\}} (x) & \text{ for } x \in \mathbb{R}^{n}.
\end{cases}
\end{equation}
By the boundedness obtained in \Cref{comparison fractional and R}, the maximum principle for $w$, and the regularity $C^\alpha(\mathbb{R}^n)$ of $w(x,0)$ (given by \Cref{cor:holderthm}), we have 
\[
\sup_{|x|^2+y^2\leq r^2}w(x,y)\leq Cr^\alpha,
\]
for a uniform constant $C>0$. Our goal is to obtain the estimative above with $1-s$ replacing $\alpha$. In particular, without loss of generality, we assume that $\alpha<1-s$.

\

We begin with the following lemma, which is the analogue of \cite[Lemma 4.5]{CaFi-08}.

\begin{lemma}\label{lem:zero-not-hull}
Let $\bar{C}>0$ and $\alpha$ be as in \Cref{holderthm} and set 
\[
\delta = \delta (\alpha, s) = \frac{1}{4}\left(\frac{\alpha}{\alpha+2s}-\frac{\alpha}{2}\right).
\] Then, there exists $r_0 > 0$, depending on $\alpha,s,\bar{C},$ and $C_0$, such that $\operatorname{co} (\Omega \cap B_{r})$ does not contain the origin for any $r \in (0, r_0)$, where
\[
\Omega \coloneqq \{ x \in \mathbb{R}^{n} ; ~ w(x,0)\geq r^{\alpha+\delta }\}
\] and $\operatorname{co} A$ stands for the convex hull of the set $A$.
\end{lemma}

\begin{proof} %%%%%What follows is a straightforward adaptation of \cite[Lemma 4.5]{CaFi-08}. We present this proof here for the sake of completeness. 
	Let $x \in \Omega$ and assume, by contradiction, that $0 \in \operatorname{co} (\Omega \cap B_r)$. By the definition of $w$, we must have $u(t,x) = \psi (x)$, or equivalently, $\tilde{v} (x,0) = 0$. Note that, for $x \in \Omega$,
	\[
	\lim_{y\rightarrow 0^+}y^a\tilde{v}_{y}(x,y)=-(-\Delta)^su(t,x)+\pazocal{R}u(t,0)=-w(x,0)\leq -r^{\alpha+\delta }.
	\] Moreover, by \Cref{lem:b1b5}\eqref{b2},
\[
\begin{split}
\tilde{v}(x,y) & = \int_0^y \tilde{v}_{y} (x,\tau) \, \dd \tau = \int_0^y \frac{1}{\tau^{a}} \left( \lim_{\rho \to 0^{+}} \rho^{a} \tilde{v}_{y} (x,\rho) + \int_{0}^{\tau} \rho^{a} \tilde{v}_{y} (x,\rho) \, \dd \rho \right) \, \dd \tau \\
    & \le - \frac{r^{\alpha+\delta }y^{2s}}{2s} + \int_0^y \frac{2 n C_0}{\tau^a} \int_0^\tau \rho^a \, \dd \rho  \, \dd \tau \\
    & = - \frac{r^{\alpha+\delta }y^{2s}}{2s} + \frac{2 n C_0}{1+a} \int_0^y \tau \, \dd \tau \\
    & = - \frac{r^{\alpha+\delta }y^{2s}}{2s} + \frac{n C_0 y^{2}}{1+a}.
\end{split} 
\] Now, by \Cref{holderthm}, we know $\tilde{v}(0,y) \ge - \bar{C} y^{\alpha+2s}$. Also, by the semiconvexity of $\tilde{v}$, given by \Cref{lem:b1b5}\eqref{b2}, we have 
\[
\tilde{v}(0,y) + \nabla_{x}\tilde{v}(0,y)\cdot x \le \tilde{v}(x,y) + C_{0} r^{2}.
\] Thus, since $0 \in \operatorname{co} (\Omega \cap B_r)$,
\[
\sup_{x\in  \operatorname{co} (\Omega \cap B_r)}\nabla_{x}\tilde{v}(0,y) \cdot x\geq -|\nabla_{x}\tilde{v}(0,y)| \inf_{\operatorname{co} (\Omega \cap B_r)}|x|=0
\] and we have
%%%%%diretamente por semiconvexidade, parece que falta o termo com gradiente em $x$: $\tilde{v}(x,y) \ge \tilde{v}(0,y) + \nabla_{x}\tilde{v}(0,y)\cdot x - C_{0} |x|^{2} \implies \tilde{v}(0,y) + \nabla_{x}\tilde{v}(0,y)\cdot x \le \tilde{v}(x,y) + C_{0} r^{2}$ -- se zero está na envoltória convexa, usa que \[\sup_{x\in  \operatorname{co} (\Omega \cap B_r)}\nabla_{x}\tilde{v}(0,y)\cdot x\geq |\nabla_{x}\tilde{v}(0,y)|\sup_{\operatorname{co} (\Omega \cap B_r)}- |x|=-|\nabla_{x}\tilde{v}(0,y)\inf_{\operatorname{co} (\Omega \cap B_r)}|x|=0\]
\[
\tilde{v}(0,y) \le \sup_{x\in\operatorname{co} (\Omega \cap B_r)} \tilde{v}(x,y)+C_0r^2.
\] Putting all these together, we have, for any $r,y \in (0,1)$,
\begin{equation}\label{contrad}
\bar{C} y^{\alpha+2s} + \frac{n C_0 }{1+a} y^{2} + C_0 r^2 \geq \frac{r^{\alpha+\delta }y^{2s}}{2s}.
\end{equation} In order to get a contradiction, we relate $y$ and $r$ by the formula $y^{\alpha} = r^{\alpha + 2\delta }$, so that \eqref{contrad} implies
\[
\bar{C}r^{\delta}+\frac{nC_0}{1+a}r^{4\delta \alpha^{-1} + \delta +\gamma} + C_0 r^{\delta +\gamma} \ge \frac{1}{2s},
\] where $\gamma \coloneqq 2 -  \alpha^{-1} (\alpha+2s)(\alpha + 2\delta)$ which is positive by the definition of $\delta$. Now, the left hand side goes to zero as $r \to 0$ and we have a contradiction for small values of $r$.
\end{proof} 
%%%%\[
%%%%r\coloneqq \min\left\{\left(\frac{1}{8s\bar{C}}\right)^{-\delta },\left(\frac{1+a}{8snC_0}\right)^{-(4\delta /\alpha+\delta +\gamma)},\left(\frac{1}{8sC_0}\right)^{-(\gamma+\delta )}\right\},
%%%%\]

We remark that $\delta<\gamma$, since $2s>1$ and $\alpha<\gamma$. The next two technical lemmas are key ingredients to prove the monotonicity formula of \Cref{lem:monotonicity}.

\begin{lemma}\label{lem:antes-monoton}
There exists $C > 0$ such that, for every $r \ge 0$,
	\[
	\limsup_{y \rightarrow 0^{+}} \int_{B_r} \frac{y^{-a} \, \partial_y \big(w(x,y)^2\big)}{(|x|^2 + y^2)^{(n-1-a)/2}} \, \dd x \ge -C r^{\alpha+1+a}.
	\] Moreover, 
	\[
	\lim_{y\rightarrow 0^{+}} \int_{B_r} \partial_y \left( \big(|x|^2 + y^2\big)^{-(n-1-a)/2} \right) y^{-a} w (x,y)^2 \, \dd x = 0.
	\]
\end{lemma}

\begin{proof}
To show the first estimate, we begin by noticing the following properties:
\begin{enumerate}[$(i)$]
\item From \Cref{comparison fractional and R}, we have $w(x,0)=0$ for $x\in\mathbb{R}^n\setminus \Lambda$ and $w(x,0)\geq0$ for $x\in\Lambda$. Hence, by the maximum principle $w(x,y)\geq 0$, that is, $w(x,y) \geq w(x,0)$ for all $x\in\mathbb{R}^n\setminus \Lambda$ and $y>0$.
\item From \Cref{lem:b1b5}, we have 
\[
y^av_y(x,y)\leq \lim_{\tau \rightarrow 0^+} \tau^a v_y(x,\tau) + \frac{nC_0}{1+a} y^{1+a},
\]
with the limit well-defined since $-(-\Delta)^su(t,x)+\pazocal{R}u(t,0)$ is Hölder continuous on $\Lambda$ and smooth outside (by \Cref{comparison fractional and R}).
\item The function $y^av_y$ is a  solution of
\[\begin{cases}
L_{-a}(y^av_y)=0;\\
\lim\limits_{y\rightarrow 0^+}y^av_y(x,y)=-(-\Delta)^su(t,x)+\pazocal{R}u(t,0).
\end{cases}\]
Moreover, from \Cref{comparison fractional and R}, we have that $w(x,0)\geq (-\Delta)^su(t,x)-\pazocal{R}u(t,0)$ and then, by the maximum principle, $w(x,y) \geq -y^av_y(x,y)$ on $\mathbb{R}^n\times\mathbb{R}^+$. Since 
\[
w(x,0)=-\lim_{y\rightarrow 0^+}y^av_y(x,y) \ \text{ in } \ \Lambda,
\] the previous item implies that, for all $x\in\Lambda$ and $y>0$,
\begin{equation}\label{eqn:w-em-Lambda}
w(x,y)\geq w(x,0)-\frac{nC_0}{1+a}y^{1+a}.
\end{equation}
\end{enumerate} 
From  $(i)$ and $(iii)$, we have that \eqref{eqn:w-em-Lambda} actually holds for all $x\in\mathbb{R}^n$ and $y>0$. Furthermore, since $w$ is non-negative and $C^\alpha_x$, we conclude
\[
w(x,y)^2 - w(x,0)^2 \ge - \frac{nC_0}{1+a} y^{1+a} [w(x,y)+w(x,0)] \ge - K y^{1+a}(r+y)^\alpha,
\] for all $x\in B_r$, $y>0$, and a uniform constant $K>0$.

We now use the change of variable $\tau(y) \coloneqq \big(\frac{y}{1+a}\big)^{1+a}$ and define $\tilde{w}(x,\tau)\coloneqq w(x,y)$. Then, the above inequality can be rewritten as
\begin{equation}\label{tildew}
\tilde{w}(x,y)^2 - \tilde{w}(x,0)^2 \ge -K' \tau(r + \tau^{1/(1+a)})^\alpha,
\end{equation} for all $x\in B_r$, $y>0$, and a uniform constant $K'>0$. Using that $y^{-a} \partial_y \big(w(x,y)^2 \big) = \partial_{\tau} \big(\tilde{w}(x,\tau)^2 \big)$, we have that
\[
\begin{split}
\limsup_{y\rightarrow 0^{+}} \int_{B_r} \frac{y^{-a} \partial_y \big(w(x,y)^2 \big)}{(|x|^2+y^2)^{(n-1-a)/2}} \, \dd x = \limsup_{s\rightarrow 0^{+}} \int_{B_r} \frac{\partial_{\tau} \big(\tilde{w}(x,\tau)^2 \big)}{(|x|^2 + (1+a)^2 \tau^{2/(1+a)})^{(n-1-a)/2}} \, \dd x .
\end{split}
\] To estimate the right hand side above, we consider the average with respect to $\tau \in [0,\epsilon]$ and we use Fubini's Theorem to obtain
\[\begin{split}
I_\epsilon & \coloneqq \frac{1}{\epsilon} \int_0^\epsilon\int_{B_r}  \frac{\partial_{\tau} \big(\tilde{w}(x,\tau)^2 \big)}{(|x|^2+(1+a)^2 \tau^{2/(1+a)})^{(n-1-a)/2}} \, \dd x \, \dd \tau \\
  & = \frac{1}{\epsilon}  \int_{B_r} \left( \frac{\tilde{w}(x,\epsilon)^2}{(|x|^2+(1+a)^2\epsilon^{2/(1+a)})^{(n-1-a)/2}}-\frac{\tilde{w}(x,0)^2}{|x|^{n-1-a}} \right) \, \dd x \\
  & \quad  - \frac{1}{\epsilon} \int_0^\epsilon \int_{B_r} \tilde{w}(x,\tau)^2  \frac{\dd}{\dd \tau} \left(|x|^2+(1+a)^2 \tau^{2/(1+a)}\right)^{-(n-1-a)/2} \, \dd x \, \dd \tau.
\end{split}
\] Observe that
\[
\frac{\dd}{\dd \tau} \left(|x|^2+(1+a)^2 \tau^{2/(1+a)}\right)^{-(n-1-a)/2} \le 0.
\] Hence, by \eqref{tildew} and the fact that $w(\cdot,0)=\tilde{w}(\cdot,0) \in C_x^\alpha$, we have 
\[\begin{split}
I_\epsilon & \ge \frac{1}{\epsilon} \int_{B_r} \left( \frac{\tilde{w}(x,0)^2 - K' \epsilon (r+\epsilon^{1/(1+a)})^\alpha}{(|x|^2 + (1+a)^2 \epsilon^{2/(1+a)})^{(n-1-a)/2}} - \frac{\tilde{w}(x,0)^2}{|x|^{n-1-a}} \right) \, \dd x \\
           & \geq \int_{B_r}\frac{-K'(r+\epsilon^{1/(1+a)})^\alpha}{|x|^{n-1-a}} \, \dd x  + \frac{C}{\epsilon} \int_{B_r} \left[ \frac{|x|^{2\alpha}}{(|x|^2 + (1+a)^2\epsilon^{2/(1+a)})^{(n-1-a)/2}} - \frac{|x|^{2\alpha}}{|x|^{n-1-a}}\right] \, \dd x \\
           & \eqqcolon I_{1\epsilon}+I_{2\epsilon}.
\end{split}\] We have
\[\lim_{\epsilon\rightarrow 0}I_{1\epsilon} = - \frac{K' n \omega_{n}}{a + 1} r^{\alpha+1+a} = -K'C_{n,a}r^{\alpha+1+a}.
\] For the second term $I_{2\epsilon}$, we split the integral over $B_{\epsilon^\beta}$ and over $B_r\setminus B_{\epsilon^\beta}$, denoting these by $I^1_{2\epsilon}$ and $I^2_{2\epsilon}$, respectively, and the exponent $\beta > 0$ is yet to be chosen. On the one hand, to estimate $I_{2\epsilon}^{1}$, we choose $\beta \in \big( \frac{1}{2\alpha+a+1}, \frac{1}{a+1} \big)$, and we have that
\[
\lim_{\epsilon \rightarrow 0} I^1_{2\epsilon} \ge - \lim_{\epsilon \rightarrow 0} \frac{C}{\epsilon} \int_{B_{\epsilon^\beta}} \frac{|x|^{2\alpha}}{|x|^{n-1-a}} \, \dd x = - \frac{C n \omega_{n}}{2 \alpha + a + 1} \, \lim_{\epsilon\rightarrow 0}  \epsilon^{\beta(2\alpha+a+1)-1} = 0.
\] On the other hand, for all $|x| \ge \epsilon^\beta$, we have $\epsilon^{2/(1+a)} \le |x|^2$, so that
\[
\left(|x|^2+(1+a)^2\epsilon^{2/(1+a)}\right)^{(n-1-a)/2}\le  C\left(|x|^{n-1-a}+C\epsilon^{2/(1+a)}|x|^{n-3-a}\right)
\] and the term $I_{2\epsilon}^{2}$ can be estimated as
\[\begin{split}
I^2_{2\epsilon} & \ge \frac{C}{\epsilon} \int_{\epsilon^\beta}^r \left[ \frac{\rho^{n-1+2\alpha}}{\rho^{n-1-a} + C\epsilon^{2/(1+a)} \rho^{n-3-a}} - \frac{\rho^{n-1+2\alpha}}{\rho^{n-1-a}}\right] \, \dd \rho \\
 & = - \frac{C}{\epsilon} \int_{\epsilon^\beta}^r \rho^{2\alpha+a} \frac{\epsilon^{2/(1+a)}}{\rho^2+C\epsilon^{2/(1+a)}} \, \dd \rho \\ 
 & \ge -\frac{C\epsilon^{2/(1+a)}}{\epsilon}\int_{\epsilon^\beta}^r\rho^{2\alpha+a-2} \, \dd \rho \\
 & \ge -C_r\epsilon^{2/(1+a)-1} \left[1+\epsilon^{\beta(2\alpha+a-1)}\right].
\end{split}\]
Recall that $2 > 1 + a$, and so we only need to consider the case $2 \alpha +a-1<0$, since otherwise we clearly have $\lim_{\epsilon\rightarrow 0} I^2_{2\epsilon}\ge 0$. Moreover, since $\beta<1/(1+a)$, we have
\[
\frac{2}{1+a}-1+\beta(2\alpha+a-1) \ge \frac{2\alpha}{1+a} > 0,
\]
which also gives $\lim_{\epsilon\rightarrow 0}I^2_{2\epsilon}\ge 0$, so that $\lim_{\epsilon\rightarrow 0}I_{2\epsilon}\ge0$. Hence, we conclude that
\[
\liminf_{\epsilon\rightarrow 0}I_\epsilon\geq -K'C_{n,a}r^{\alpha+1+a}.
\] From this, we deduce
\[\begin{split}
\limsup_{\epsilon\rightarrow 0} \int_{B_r} \frac{\partial_\tau \big(\tilde{w}(x,\epsilon)^2\big)}{(|x|^2 + (1+a)^2 \epsilon^{2/(1+a)})^{(n-1-a)/2}} \, \dd x 
\geq \liminf_{\epsilon\rightarrow 0}I_\epsilon\geq -K'C_{n,a}r^{\alpha+1+a},
\end{split}\] which is what we wanted.

To show the second claim of the lemma, we observe that, by the $C_x^\alpha$-regularity of $w$, we have
\[\begin{split}
\left| \int_{B_r} \partial_y \left( \big(|x|^2 + y^2\big)^{-(n-1-a)/2} \right) y^{-a} w (x,y)^2 \, \dd x \right| & \leq \int_{B_r}\frac{y^{1-a}}{(|x|^2+y^2)^{(n+1-a)/2-\alpha}}\, \, \dd x \\
 & \leq Cy^{1-a}\int_0^r\frac{\rho^{n-1}}{(\rho^2+y^2)^{(n+1-a)/2-\alpha}} \, \dd \rho \\
 & \leq Cy^{1-a}\int_0^r\frac{\rho^{n-1}}{(\rho+y)^{n+1-a-2\alpha}} \, \dd \rho  \\ 
 & \leq C_r\frac{y^{1-a}}{y^{1-a-2\alpha}}=C_ry^{2\alpha},
\end{split}\]
which gives
\[
\lim_{y\rightarrow 0^+} \int_{B_r} \partial_y \left( \big(|x|^2 + y^2\big)^{-(n-1-a)/2} \right) y^{-a} w (x,y)^2 \, \dd x = 0. \qedhere
\]
\end{proof}

The next lemma is the result \cite[Lemma 4.7]{CaFi-08} on the first eigenvalue of a weighted Laplacian on the half-sphere. The result applies to our modified function $w$ as proved below.

Let us denote by $\mathbb{S}^{n} \subset \mathbb{R}^{n+1}$ the $n$-dimensional sphere, and set
\[
\mathbb{S}^n_+ \coloneqq \mathbb{S}^n \cap \{x_{n+1} \ge 0\}.%; \qquad \mathbb{S}^n_0 \coloneqq \partial\mathbb{S}^n_+; \qquad \mathbb{S}^n_{0,+}\coloneqq \mathbb{S}^n_0\cap\{x_n\geq 0\}.
\] Let us also denote
\[
\mathcal{H}^{1/2}_{\, 0} \coloneqq \Big\{ h\in H^{1/2}\big(\partial(\mathbb{S}^n_+)\big); ~ h = 0 \text{ on } \partial(\mathbb{S}^n_+) \cap \{x_{n+1} = 0\} \cap \{x_n\geq 0\} \Big\}.
\] In other words, $h \in \mathcal{H}^{1/2}_{\, 0}$ when it is Sobolev in the boundary $\partial \mathbb{S}^n_+ \simeq \mathbb{S}^{n-1}$  of the upper sphere, and it vanishes on the upper part of the $(n-1)$-dimensional sphere $\partial(\mathbb{S}^n_+) \cap \{x_{n+1} = 0\}$.

\begin{lemma} \cite[Lemma 4.7]{CaFi-08} \label{lem:eigenvalue}
We have
\[
\inf_{h \in \mathcal{H}^{1/2}_{\, 0}} \frac{\displaystyle \int_{\mathbb{S}^n_+}|\nabla_\theta h|^2y^{-a}\, \dd \sigma }{\displaystyle \int_{\mathbb{S}^n_+}h^2y^{-a}\, \dd \sigma }=(1-s)(n-1+s),
\] where $\nabla_\theta$ is the derivative with respect to the angular variables.
\end{lemma}

\begin{proof} For convenience of the reader, we reproduce the proof by Caffarelli and Figalli. Let 
\[
\bar{H}(x,y)\coloneqq(\sqrt{x_n^2+y^2}-x_n)^{1-s}
\] and denote by $\bar{h}(\theta)$ its restriction to $\mathbb{S}^n_+$, which gives $\bar{H}=r^{1-s}\bar{h}(\theta)$. As shown in \cite[Proposition 5.4]{SEE!!}, $\bar{h}$ is the first eingenfunction related to the minimization problem above. If $\lambda_1$ is the correspoding eigenvalue, our goal is to show that $\lambda_{1} = - (1-s)(n-1+s)$.

First, we claim that $\bar{H}$ satisfies $L_{-a} \bar{H} = 0$ for $y > 0$. Indeed, the function
\[
G(x_{n}, y) \coloneqq (\sqrt{x_n^2+y^2}-x_n)^{1/2}
\] is harmonic in $y > 0$ as the imaginary part of $z \longmapsto z^{1/2}$. Since $\bar{H} = G^{1+a}$, direct computation yields
\[
L_{-a} \bar{H} = L_{-a} G^{1+a} = (1+a) y^{-a} G^{a} \Delta_{x,y} G + (1+a) a y^{-a}  G^{a-1} \left( |\nabla_{x,y} G|^{2} - \frac{G G_{y}}{y} \right) = 0.
\]
Next, since $\bar{h}$ is an eigenfunction, we have $\operatorname{div}_\theta (y^{-a} \nabla_\theta \bar{h}) = \lambda_1 \bar{h}$. In particular,
\[
\Delta_\theta\bar{h}(0,1)=\lambda_1\bar{h}(0,1).
\] Moreover, in spherical coordinates,
\[
0=L_{-a}\bar{H}=\Delta_r\bar{H}+\frac{n}{r}\bar{H}_r+\frac{1}{r^2}\Delta_\theta\bar{H}-\frac{a}{y}\bar{H}_y,
\] and we obtain
\[\begin{split}
0 & = \Delta_r\bar{H}(0,1)+n\bar{H}_r(0,1)+\Delta_\theta\bar{H}(0,1)-a\bar{H}_y(0,1) \\
  & = - (1-s)s\bar{h}(0,1)+(1-s)(n-a)\bar{h}(0,1)+\Delta_\theta\bar{h}(0,1).
\end{split}\]
Therefore,
\[
\lambda_1\bar{h}(0,1)=\Delta_\theta\bar{h}(0,1)=-(1-s)(n-1+s)\bar{h}(0,1).  \qedhere
\]
\end{proof} 

We now prove the monotonicity formula: the result and its proof are found in \cite[Lemma 4.8]{CaFi-08}. For the convenience of the reader, we reproduce the proof. 

\begin{lemma}[Monotonicity Formula]\label{lem:monotonicity}
Let $w$ be given by \eqref{eqn:w} and denote
\[
B_r^+ \coloneqq \{ z=(x,y)\in \mathbb{R}^n \times \mathbb{R}^+ ; ~ |z|<r \}.
\] For $r \in (0,1]$, define
\[
\varphi(r)\coloneqq \frac{1}{r^{2(1-s)}}\int_{B_r^+}\frac{|\nabla_{z} w(z)|^2y^{-a}}{|z|^{n-1-a}}\, \dd z. 
\] Then, there exists $C>0$ such that, for all $r \in (0,1]$,
\[
\varphi(r)\leq C \left( 1 + r^{2\alpha + \delta - a - 1} \right).
\]
\end{lemma}

\begin{proof} 
Set 
\[
\varphi_\epsilon(r) \coloneqq \frac{1}{r^{2(1-s)}}\int_{B_r^+\cap \{y>\epsilon\}}\frac{|\nabla_{z} w(z)|^2y^{-a}}{|z|^{n-1-a}}\, \dd z.
\] By the Monotone Convergence Theorem, we can bound $\varphi$ by $\liminf_{\epsilon\rightarrow 0}\varphi_\epsilon$. Moreover, we note that $\varphi(r)$ is bounded by $\varphi(1)$. Hence, we only need to bound $\liminf_{\epsilon\rightarrow 0}\varphi_\epsilon(1)$. Let $\chi: \mathbb{R}^n\rightarrow [0,1]$ be a smooth compactly supported function with $\chi \equiv 1$ in $B_1\subset \mathbb{R}^n$. Thus, 
\[
\varphi_\epsilon(r) \leq \int_\epsilon^1\int_{\mathbb{R}^n}\frac{|\nabla_{z} w(z)|^2y^{-a}}{|z|^{n-1-a}} \chi(x) \, \dd x \, \dd y.
\]
The definition of $w$ in \eqref{eqn:w} gives $L_{-a} w = 0$ and so we have $L_{-a}(w^2)=2|\nabla_{z} w|^2y^{-a}$. Then, integration by parts gives
\[
\begin{split}
\varphi_\epsilon(r) & \leq -\int_\epsilon^1\int_{\mathbb{R}^n} \nabla_{z} (w^2) \cdot \nabla_{z} \left( \frac{1}{2|z|^{n-1-a}}\right)y^{-a}\chi(x) \, \dd x \, \dd y\\
& \quad -\int_\epsilon^1\int_{\mathbb{R}^n}\nabla_x (w^2)\cdot \nabla_x\chi(x) \frac{y^{-a}}{2|z|^{n-1-a}} \, \dd x \, \dd y+\left.\int_{\mathbb{R}^n}\partial_y(w^2)\frac{y^{-a}}{2|z|^{n-1-a}} \, \dd x \right\vert_{y=\epsilon}^{y=1}.
\end{split}
\] Using that $L_{-a} |z|^{-n+1+a} = C \delta_{(0,0)}$, we can integrate by parts once more to obtain
\[
\begin{split}
\varphi_\epsilon(r) & \leq \int_\epsilon^1\int_{\mathbb{R}^n} w^2 \Delta_x\chi(x) \frac{y^{-a}}{2|z|^{n-1-a}} \, \dd x \, \dd y + \int_\epsilon^1\int_{\mathbb{R}^n} w^2 \nabla_x\chi(x)\cdot\nabla_x\left(\frac{1}{|z|^{n-1-a}}\right)y^{-a} \, \dd x \, \dd y \\
       & \quad - \left.\int_{\mathbb{R}^n} w^2 \partial_y\left(\frac{1}{2|z|^{n-1-a}}\right)y^{-a}\chi(x) \, \dd x \right\vert_{y=\epsilon}^{y=1}
+\left.\int_{\mathbb{R}^n} \partial_y(w^2)\frac{y^{-a}\chi(x)}{2|z|^{n-1-a}} \, \dd x \right\vert_{y=\epsilon}^{y=1},
\end{split}
\]
since $(0,0)\notin [\epsilon,1]\times \mathbb{R}^n$. Recall that $\chi\equiv 1$ in $B_1$, that $w$ is of class $C^\alpha_x$, and that $\operatorname{supp} \chi \subseteq B_R$ for some $R>0$, so that
\[\begin{split}
 \int_\epsilon^1 \int_{\mathbb{R}^n} \left( w^2 \Delta_x\chi(x) \frac{y^{-a}}{2|z|^{n-1-a}}+ w^2 \nabla_x\chi(x)\cdot\nabla_x\left(\frac{1}{|z|^{n-1-a}}\right)y^{-a} \right) \, \dd x \, \dd y\\
 \leq C \int_\epsilon^1  y^{-a} \int_1^R \left( |r^2+y^2|^{\alpha+a/2}+|r^2+y^2|^{\alpha+(1+a)/2} \right) \, \dd r \, \dd y < + \infty.
\end{split}\]
Moreover, since $w$ is smooth for $y>0$, we obtain
\[
-\left.\int_{\mathbb{R}^n} w^2 \partial_y\left(\frac{1}{2|z|^{n-1-a}}\right)y^{-a}\chi(x) \, \dd x \right\vert_{y=1}
+\left.\int_{\mathbb{R}^n} \partial_y(w^2)\frac{y^{-a}\chi(x)}{2|z|^{n-1-a}} \, \dd x \right\vert_{y=1}< + \infty.
\]
Using \Cref{lem:antes-monoton}, we conclude that 
\[
\varphi(r)\leq \liminf_{\epsilon\rightarrow 0}\varphi_\epsilon(1) \le C.
\] Hence, we have that $\varphi_\epsilon(r) \longrightarrow \varphi(r)$ locally uniformly in $(0,1]$.
This shows, in particular, that $\varphi(r)$ is well-defined. Now, take $\epsilon<r$ and use again that $L_{-a} |z|^{-n+1+a} = C \delta_{(0,0)}$ to obtain
\[
\begin{split}
\varphi'_\epsilon(r) & = - \frac{1-s}{r^{3-2s}} \int_{B_r^+\cap\{y>\epsilon\}} \frac{L_{-a}(w^2)}{|z|^{n-1-a}} \, \dd z + \frac{1}{r^n} \int_{\partial B_r^+\cap\{y>\epsilon\}} |\nabla_{z} w(z)|^2 y^{-a} \, \dd \sigma \\
     & = - \frac{2(1-s)}{r^{1+2(1-s)}} \int_{\partial(B_r^+\cap\{y>\epsilon\})}w\nabla_{z} w \cdot \nu \frac{y^{-a}}{|z|^{n-1-a}} \, \dd \sigma \\
     & \, \quad + \frac{1-s}{r^{1+2(1-s)}} \int_{B_r^+\cap\{y>\epsilon\}} \nabla_{z}   (w^2) \cdot \nabla_{z}   \left(\frac{1}{|z|^{n-1-a}}\right) y^{-a} \, \dd z \\   
     & \, \quad + \frac{1}{r^n}\int_{\partial B_r^+\cap\{y>\epsilon\}}|\nabla_{z} w(z)|^2y^{-a}\, \dd \sigma \eqqcolon A_\epsilon+B_\epsilon+C_\epsilon.
\end{split}
\] We estimate each of these three terms. By \Cref{lem:antes-monoton} and the Cauchy-Schwarz's inequality, 
\[
\begin{split}
\lim_{\epsilon\rightarrow 0^+}A_\epsilon & = - \frac{1-s}{r^{n+1}}\lim_{\epsilon\rightarrow 0^+}\int_{\partial B_r^+\cap\{y>\epsilon\}}(w^2)_ry^{-a}\, \dd \sigma + \frac{1-s}{r^{1+2(1-s)}}\lim_{\epsilon\rightarrow 0^+}\int_{B_r^+\cap\{y=\epsilon\}}(w^2)_y\frac{y^{-a}}{|z|^{n-1-a}}\, \dd \sigma 
 \\
 &
\geq -\frac{(1-s)^2}{r^{n+2}}\int_{\partial B_{r,+}}w^2 y^{-a}\, \dd \sigma -\frac{1}{r^n}\int_{\partial B_{r,+}}(w_r)^2 y^{-a}\, \dd \sigma -Cr^{\alpha-1} \\
&
\geq-\frac{(1-s)^2}{r^{n+2}}\int_{\partial B_{r,+}}w^2 y^{-a}\, \dd \sigma -\frac{1}{r^n}\int_{\partial B_{r,+}}|\nabla_{z} w(z)|^2y^{-a}\, \dd \sigma 
 \\
 & \, \quad +\frac{1}{r^{n+2}}\int_{\partial B_{r,+}}|\nabla_\theta w|^2 y^{-a}\, \dd \sigma -Cr^{\alpha-1}.
\end{split}
\] Also,
\[ 
\begin{split}
\lim_{\epsilon\rightarrow 0^+}B_\epsilon & = - \frac{(1-s)(n-1-a)}{r^{n+2}} \lim_{\epsilon\rightarrow 0^+} \int_{\partial B_r^+\cap\{y>\epsilon\}}w^2y^{-a}\, \dd \sigma \\
& \, \quad  - \frac{1-s}{r^{1+2(1-s)}}\lim_{\epsilon\rightarrow 0^+} \int_{B_r^+\cap\{y=\epsilon\}} w^2 \partial_y \left(\frac{1}{|z|^{n-1-a}}\right)y^{-a}\, \dd \sigma \\ 
     & = - \frac{(1-s)(n-1-a)}{r^{n+2}}\int_{\partial B_{r,+}}w^2y^{-a}\, \dd \sigma.
\end{split}
\]
Hence, using that $\varphi_\epsilon$ converges uniformly as $\epsilon\rightarrow 0^+$, we have that the distributional derivative $D_r\varphi$ satisfies
\[\begin{split}
D_r\varphi\geq \frac{1}{r^{n+2}}\int_{\partial B_{r,+}}|\nabla_\theta w|^2 y^{-a}\, \dd \sigma -Cr^{\alpha-1}
+\frac{\lambda_1}{r^{n+2}}\int_{\partial B_{r,+}}w^2y^{-a}\, \dd \sigma ,
\end{split}\]
where $\lambda_1$ as in (the proof of) \Cref{lem:eigenvalue}. Consider $\bar{W}\coloneqq (w-r^{\alpha+\delta})^-$. By \Cref{lem:zero-not-hull}, $\bar{W}$ is admissible for the eigenvalue problem in \Cref{lem:eigenvalue}. We compute 
\[
|\nabla_\theta \bar{W}|^2\leq|\nabla_\theta w|^2 \ \text{ and } \ (w-\bar{W})^2 + 2\bar{W}(w-\bar{W}) + \bar{W}^2=w^2,
\] to conclude that
\[\begin{split}
D_r\varphi\geq \frac{\lambda_1}{r^{n+2}}\int_{\partial B_{r,+}} \big[(w-\bar{W})^2+2\bar{W}(w-\bar{W})\big] \, y^{-a}\, \dd \sigma -Cr^{\alpha-1}
\geq-Cr^{2\alpha+\delta-a-2}-Cr^{\alpha-1},
\end{split}\]
since $|\bar{W}|\leq|w|\leq Cr^\alpha$ and $|w-\bar{W}|\leq r^{\alpha+\delta}$. Therefore, integration in the interval $[r,1]$ yields
\[
\varphi(r)\leq \varphi(1)+Cr^{2\alpha+\delta-a-1}+C\leq C(1+r^{2\alpha+\delta-a-1}) 
\]
for all $r \in (0, 1]$, since $1+a>0$ and $\varphi(1)$ is universally bounded.
\end{proof}

Now, we we are able to obtain the optimal modulus of continuity of $w$. In particular, we obtain an improved regularity and the optimal regularity of the lower order and free boundary terms, respectively.

\begin{proposition}\label{prop:optimal-decay}
Let $u$ be the solution of \eqref{prin}. Then,
\[
[(-\Delta)^su-\pazocal{R}u]\chi_{\{u=\psi\}}\in L^\infty((0,T];C^{1-s}(\mathbb{R}^n))\quad \text{and} \quad \pazocal{R}u \in L^\infty((0,T];C^{1-s+\gamma}(\mathbb{R}^n)).
\]
\end{proposition}

\begin{proof} 
Let $\eta_\epsilon$ be a mollifier, define $w_\epsilon \coloneqq \eta_\epsilon \ast w$, and observe $L_{-a} w_\epsilon = (L_{-a} w) \ast \eta_\epsilon = 0$. Moreover,
\[
w_\epsilon(x,y)-w_\epsilon(x,0)\geq -\frac{nC_0}{1+a}y^{1+a}.
\] Set $\bar{W}_\epsilon\coloneqq (w_\epsilon-r^{\alpha+\delta})^+$, which satisfies  $L_{-a} \bar{W}_\epsilon \le 0$ in the set $\{y>0\}$ and 
\[
\bar{W}_\epsilon(x,y)-\bar{W}_\epsilon(x,0)\geq -\frac{nC_0}{1+a}y^{1+a}.
\] We now consider, for $(x,y) \in \mathbb{R}^n\times\mathbb{R}$,
\[
\tilde{w}_\epsilon(x,y)\coloneqq \bar{W}_\epsilon(x,|y|)+\left(1+\frac{nC_0}{1+a}\right)|y|^{1+a}.
\] Note that $L_{-a} \tilde{w}_\epsilon \le 0$ in the set $\{y \neq 0\}$, and \[
\tilde{w}_\epsilon(x,y)-\tilde{w}_\epsilon(x,0)\geq|y|^{1+a}.
\] Since $\tilde{w}_\epsilon$ is smooth in $x$, we conclude $\tilde{w}_\epsilon$ is a subsolution for $L_{-a}$ in the whole $\mathbb{R}^n\times\mathbb{R}$. Then, let $\epsilon\rightarrow 0$ so that
\[
\tilde{w}(x,y)\coloneqq (w(x,|y|)-r^{\alpha+\delta})^++\left(1+\frac{nC_0}{1+a}\right)|y|^{1+a}
\] is a subsolution globally. By \Cref{lem:zero-not-hull}, the convex hull of the set where $\tilde{w} (\cdot,0) \ge 0$ does not contain the origin and it is thus contained in ``some half'' of $B_r$. In particular, $\tilde{w} (\cdot,0) \equiv 0$ in a set which is bigger than the other half of $B_{r}$. So, by a weighted Poincaré inequality (see \cite[Theorem 1.5]{poincare}) and the definition of $\varphi$ (see \Cref{lem:monotonicity}), we obtain, for all $r \in (0, 1]$,
\[
\begin{split}
\int_{ B^+_r}\tilde{w} (z)^2y^{-a} \, \dd z & \leq Cr^2 \int_{ B^+_r}|\nabla_{z}\tilde{w} (z)|^2y^{-a} \, \dd z \\ 
    & \leq Cr^2\left[\int_{ B^+_r}|\nabla_{z} w (z)|^2y^{-a} \, \dd z +r^{n+1+a}\right] \\
    & \leq Cr^{n+2} \big( \varphi(r)+r^{1+a} \big) \\
    & \leq Cr^{n+2} \big( 1+\varphi(r) \big),
\end{split}
\] since $|\nabla_{z}\tilde{w}|^2 \leq |\nabla_{z} w|^2+C|y|^{2a}$. Then, since $\tilde{w}$ is $L_{-a}$-subharmonic, we use \Cref{lem:monotonicity} and get
\[
\sup_{B^+_{r/2}} \tilde{w}^2 \le \frac{C}{r^{n+1-a}} \int_{ B^+_r} \tilde{w}(z)^2 y^{-a} \, \dd z \leq C \big(r^{1+a}+r^{2\alpha+\delta}\big).
\]

Hence, for all $r \in (0, 1]$,
\begin{equation}\label{iteratemonotonicity}
\sup_{B_r} w \le C \Big( \sup_{B^+_{r/2}} \tilde{w} + r^{\alpha + \delta} + r^{1+a} \Big) \le C \big( r^{1-s}+r^{\alpha+\delta/2} \big).
\end{equation} We conclude by the same argument as \Cref{cor:holderthm} that  
\[
\|w\|_{C_x^{\beta_\alpha}(\mathbb{R}^n)}\leq C, \quad \text{ where } \quad \beta_\alpha \coloneqq \min \{1-s, \alpha + \delta/2\}.
\] 
As remarked previously, $\delta<\gamma$, thus $\beta_\alpha<\gamma$. Hence, by \eqref{initialregr}, we have
\[
\big[(-\Delta)^su(t, \cdot)-\pazocal{R}u(t, \cdot)\big] \chi_{\{u(t, \cdot)=\psi\}}\in C^{\beta_\alpha} (\mathbb{R}^n).
\] Therefore, we have
\[
\partial_t u+(-\Delta)^su = \big[(-\Delta)^su-\pazocal{R}u \big] \chi_{\{u =\psi\}}+\pazocal{R}u\in L^\infty((0,T];C^{\beta_\alpha}(\mathbb{R}^n)).
\]
Hence, by \eqref{regularityheateq}, $(-\Delta)^su(t,\cdot)\in C^{\beta_\alpha-0^+}(\mathbb{R}^n)$, and by \cite[Proposition 2.1.8]{silvestre}, we have $u(t,\cdot)\in C^{\beta_\alpha+2s-0^+}(\mathbb{R}^n)$, thus 
\begin{equation}\label{regularityR}
\pazocal{R}u \in L^\infty((0,T];C^{\beta_\alpha+\gamma}(\mathbb{R}^n)).
\end{equation}

By the definition of $\delta$ (see \Cref{lem:zero-not-hull}), given $\alpha_0>0$, there exists $\delta_0>0$ such that $\delta\geq \delta_0>0$ for all $\alpha'\in[\alpha_0,1-s]$. If $\beta_\alpha=1-s$, the proposition follows. Otherwise, we apply the monotonicity formula  (as in \eqref{iteratemonotonicity}) $k$ times and the argument above to obtain 
\[
\sup_{B_r} w \le C \big( r^{1-s}+r^{\alpha+k\,\delta_0/2} \big), \quad \pazocal{R}u \in L^\infty((0,T];C^{\min \{1-s, \alpha +k\, \delta_0/2\}+\gamma}(\mathbb{R}^n)). 
\]
Choosing $k$ large enough and using the same argument as \Cref{cor:holderthm}, the proposition follows.
\end{proof}
\section{Almost optimal regularity in time}
We note that \Cref{prop:optimal-decay} implies that 
\[
\partial_t u+(-\Delta)^su = \big[(-\Delta)^su-\pazocal{R}u\big] \chi_{\{u=\psi\}}+\pazocal{R}u \in L^\infty((0,T];C^{1-s}(\mathbb{R}^n)),
\]
which gives (see \eqref{regularityheateq})
\begin{equation}\label{regtx}
\partial_t u \ \text{ and } \ (-\Delta)^su \ \in \ L^\infty((0,T];C^{1-s-0^+}(\mathbb{R}^n)).
\end{equation} From this, we are able to show the first step of the iteration procedure that eventually grants us the optimal regularity of the solution. We remark that by \eqref{gamma} and \Cref{prop:optimal-decay}, we have
\[
\pazocal{R}u\in L^\infty((0,T];C^{1-s+\gamma}(\mathbb{R}^n)).
\]

\begin{lemma}\label{lem:iterate-first}
We have $\big[(-\Delta)^su-\pazocal{R}u\big] \chi_{\{u=\psi\}}\in C_{t,x}^{\frac{1-s}{1+s}-0^+,1-s}((0,T]\times \mathbb{R}^n)$.
\end{lemma}

\begin{proof}
We need to estimate
\[
\big[(-\Delta)^su-\pazocal{R}u\big](t,x) \chi_{\{u(t,\cdot)=\psi\}}-\big[(-\Delta)^su-\pazocal{R}u\big](s,x) \chi_{\{u(s,\cdot)=\psi\}}.
\]
We notice that we only need to consider $x\in \{u(\tau,\cdot)=\psi\}$ for some $\tau$ (otherwise the expression vanishes). Let $0<s<t\leq T$. By \Cref{lem:semiconvexity}, $\{u(t, \cdot)=\psi\}\subseteq \{u(s, \cdot)=\psi\}$, we can assume, without loss of generality, that $x\in \{u(s, \cdot)=\psi\}$. If $x\in \{u(s, \cdot)=\psi\} \setminus \{u(t, \cdot) = \psi\}$, by \Cref{comparison fractional and R}, the left hand side below vanishes and we can find $\tau \in (s,t)$ such that $x\in\partial\{u(\tau, \cdot)=\psi\}$
\[
\big[(-\Delta)^su-\pazocal{R}u\big](\tau,x)\chi_{\{u(\tau, \cdot)=\psi\}} = \big[(-\Delta)^su-\pazocal{R}u\big](t,x)\chi_{\{u(t, \cdot)=\psi\}}. 
\] Then, we can estimate the free boundary part replacing $t$ with $\tau$. Hence, we need only consider $x\in \{u(t, \cdot)=\psi\}$. In other words, we only need to estimate both terms
\begin{equation}\label{tobeestimated}
\big|(-\Delta)^su(t,x)-(-\Delta)^su(s,x)\big|, 
\end{equation}
\begin{equation}\label{tobeestimated1}
\big|\pazocal{R}u(t,x)-\pazocal{R}u(s,x)\big|. 
\end{equation}

By the same strategy as in \cite[Lemma 4.12]{CaFi-08}, we bound the \eqref{tobeestimated} by
\[
\begin{split}
 &\left| \int_{\mathbb{R}^n} \big[(-\Delta)^su(t,x)- (-\Delta)^su(t,z)\big]\phi_r(x-z) \, \dd z \right| \\ 
   &+ \left|\int_{\mathbb{R}^n} \big[(-\Delta)^su(t,z)- (-\Delta)^su(s,z) \big] \phi_r(x-z) \, \dd z \, \right| \\
   &+ \left| \int_{\mathbb{R}^n} \big[(-\Delta)^su(s,x)- (-\Delta)^su(s,z)\big] \phi_r(x-z) \, \dd z \right|,
\end{split}
\] 
where $\phi$ is a normalized smooth cutoff function supported in $B_1$, and $\phi_r(x)\coloneqq r^{-n}\phi(x/r)$. Then, the first and third terms can be controlled by $Cr^{1-s-0^+}$. For the second term, we integrate by parts $(-\Delta)^s$ and recall the Lipschitz-in-time regularity of $u$ (see \Cref{cor:lipschitz}), so that
\[
\big|(-\Delta)^su(t,x)-(-\Delta)^su(s,x)\big|\leq C\left(r^{1-s-0^+}+\frac{(t-s)}{r^{2s}}\right).
\]
The choice $r\coloneqq (t-s)^{1/(1+s)}$ thus yields
\[
\big|(-\Delta)^su(t,x)-(-\Delta)^su(s,x)\big| \leq C(t-s)^{\frac{1-s}{1+s}-0^+}.
\]
Analogously, we bound \eqref{tobeestimated1} by
\[
\begin{split}
 &\left| \int_{\mathbb{R}^n} \big[\pazocal{R}u(t,x)- \pazocal{R}u(t,z)\big]\phi_r(x-z) \, \dd z \right| \\ 
   &+ \left|\int_{\mathbb{R}^n} \big[\pazocal{R}u(t,z)- \pazocal{R}u(s,z) \big] \phi_r(x-z) \, \dd z \, \right| \\
   &+ \left| \int_{\mathbb{R}^n} \big[\pazocal{R}u(s,x)- \pazocal{R}u(s,z)\big] \phi_r(x-z) \, \dd z. \right|,
\end{split}
\]
By the space regularity of $\pazocal{R}u$ (see \eqref{regularityR}), the first and third terms can be controlled by $C r^{1-s+\gamma}$. By \Cref{cor:lipschitz} and performing an integration by parts, we can bound the second term by
\[
C\left((t-s)+\frac{(t-s)}{r}+\left|\int_{\mathbb{R}^n} M^+_{\mathcal{L}_0}(u(t,z)-u(s,z)) \phi_r(x-z) \, \dd z \, \right|\right).
\]
Recalling that $M^+_{\mathcal{L}_0}v\coloneqq \sup_{L\in \mathcal{L}_0}L v$, we have that for all $\epsilon>0$, $M^+_{\mathcal{L}_0}v-\epsilon\leq L v$. Since $\|\phi_r\|_{L^1(B_r)}=1$ and the fact that $L$ is an integrable by parts operator, we obtain
\[
\left|\int_{\mathbb{R}^n} M^+_{\mathcal{L}_0}(u(t,z)-u(s,z)) \phi_r(x-z) \, \dd z \, \right|\leq \int_{\mathbb{R}^n}|u(t,z)-u(s,z)|| M^+_{\mathcal{L}_0}\phi_r(x-z)|\, \dd z+\epsilon.
\]
Now, by the explicit form of Pucci operator (see \eqref{K}) and the Lipschitz in time regularity of $u$, we can bound the integral above by $\frac{(t-s)}{r^{2\sigma}}$ after letting $\epsilon\longrightarrow 0^+$. Thus,
\[
\big|\pazocal{R}u(t,x)-\pazocal{R}u(s,x)\big|\leq C\left(r^{1-s+\gamma}+(t-s)+\frac{(t-s)}{r}+\frac{(t-s)}{r^{2\sigma}}\right).
\]
Assume without loss of generality that $r\leq 1$ (otherwise, we have Lipschitz regularity in time). Now, if $2\sigma\leq 1$, we have $\gamma=s-1/2$, and so
\[
\big|\pazocal{R}u(t,x)-\pazocal{R}u(s,x) \big| \le C\left(r^{1/2}+(t-s)+\frac{(t-s)}{r}\right),
\] and so the choice $r\coloneqq(t-s)^{2/3}$ gives
\[
\big|\pazocal{R}u(t,x)-\pazocal{R}u(s,x)\big| \le C(t-s)^{1/3}.
\] 
Now, if $2\sigma> 1$, we have $\gamma=s-\sigma$ and so
\[
\big|\pazocal{R}u(t,x)-\pazocal{R}u(s,x) \big| \le C\left(r^{1-\sigma}+(t-s)+\frac{(t-s)}{r^{2\sigma}}\right),
\]
and so the choice $r\coloneqq (t-s)^{1/(1+\sigma)}$ gives
\[
\big|\pazocal{R}u(t,x)-\pazocal{R}u(s,x)\big| \le C(t-s)^{\frac{1-\sigma}{1+\sigma}}.
\] 
 Finally, observe that since $s>1/2$,
\begin{equation}\label{exp}
\frac{1}{3}>\frac{1-s}{1+s} \ \ \text{ and } \ \ \frac{1-\sigma}{1+\sigma}> \frac{1-s}{1+s}. \qedhere
\end{equation}
\end{proof}

Thus, by \eqref{regalphabeta}, \Cref{lem:iterate-first} and interpolation inequalities (see, for instance, \cite[Lemma 6.32]{GiTr-01}), we obtain
\begin{equation}\label{firstregintime}
\partial_t u \ \text{ and } \ (-\Delta)^su \ \in \ C_{t,x}^{\frac{1-s}{1+s}-0^+,1-s}((0,T]\times\mathbb{R}^n).
\end{equation}
The next lemma is the key ingredient to create a bootstrap in \Cref{maintheorem}.
\begin{lemma}\label{improvedreg} 
Let $\alpha\in (0, \frac{1-s}{2s})$ and assume
\[
\partial_t u \in C_{t,x}^{\alpha,1-s}((0,T] \times \mathbb{R}^n) \ \text{ and } \ (-\Delta)^s u \in L^\infty((0,T];C^{1-s}(\mathbb{R}^n)).
\] Then, with a uniform bound,
\[
\big[(-\Delta)^su-\pazocal{R}u\big] \chi_{\{u=\psi\}} \ \in \ C_{t,x}^{(1+\alpha) \frac{1-s}{1+s},1-s} ((0,T]\times \mathbb{R}^n).
\]
\end{lemma}

\begin{proof}
As in the proof of the previous lemma, we first consider $x\in\{u(\tau,\cdot)=\psi\}$ for some $\tau$, thus without loss of generality, $x \in \{u(t, \cdot) = \psi\} \subseteq \{u(s, \cdot)=\psi\}$, where $0<s<t<T$, and estimate
\[
\begin{split}
\big|(-\Delta)^su(t,x)-(-\Delta)^su(s,x)\big| & \le \left| \int_{\mathbb{R}^n} \big[(-\Delta)^su(t,x) - (-\Delta)^su(t,z) \big]\phi_r(x-z) \, \dd z \right| \\ 
& \quad + \left|\int_{\mathbb{R}^n} \big[(-\Delta)^su(t,z) - (-\Delta)^s u(s,z) \big] \phi_r(x-z) \, \dd z \, \right| \\
& \quad + \left| \int_{\mathbb{R}^n} \big[(-\Delta)^su(s,x) - (-\Delta)^s u(s,z)\big] \phi_r(x-z) \, \dd z \right|.
\end{split}
\] Again, the first and third terms can be controlled by $Cr^{1-s}$. The second term we integrate by parts to obtain
\[
\begin{split}
\left| \int_{\mathbb{R}^n} \big[(-\Delta)^su(t,z) -  (-\Delta)^s u(s,z) \big]  \phi_r(x-z) \, \dd z \, \right|  \le  \left( \int_{\mathbb{R}^n} \big|\partial_t u(s,z)\big| \, \big|(-\Delta)^s\phi_r(x-z) \big| \, \dd z \right)(t-s) \\
+  \left|\int_{\mathbb{R}^n} \big[ u(t,z)-u(s,z)-\partial_t u(s,z) (t-s) \big] \, (-\Delta)^s\phi_r(x-z) \, \dd z \right|.
\end{split}
\] Since $\partial_t u (\cdot, z)$ is of class $C^\alpha$ in time and $\|(-\Delta)^s\phi_r\|_{L^1(B_r)}\leq C/r^{2s}$, the last term on the right hand side above is bounded by $C(t-s)^{1+\alpha}/r^{2s}$. For the integral in the first term, we use that $\partial_t u$ is of class $C^{1-s}$ in space and that $\partial_t u$ vanishes at $(t,x) \in \{u=\psi\}$ to show it is bounded by
\begin{equation}\label{boundut}
C \int_{\mathbb{R}^n} \min\{|x-z|^{1-s},1\} \, |\;\! (-\Delta)^s\phi_r(x-z)| \, \dd z.
\end{equation}
Since $\phi$ is compactly supported, $|(-\Delta)^s \phi(w)|\leq C |w|^{-n-2s}$ when $|w|$ is large enough. Hence, scaling yields, for all $w\in\mathbb{R}^n$,
\[
|(-\Delta)^s\phi_r(w)|\leq \frac{C}{r^{n+2s}+|w|^{n+2s}}.
\] Thus, \eqref{boundut} can be controled, up to a constant, by
\[\begin{split} 
\int_{B_1}\frac{|w|^{1-s}}{r^{n+2s}+|w|^{n+2s}} \, \dd w + \int_{\mathbb{R}^n\setminus B_1}\frac{1}{|w|^{n+2s}} \, \dd w \leq \frac{C}{r^{n+2s}}\int_{B_r}|w|^{1-s} \, \dd w + C \int_{B_1\setminus B_r}|w|^{1-3s-n} \, \dd w + C.
\end{split}\] This implies
\[
\int_{\mathbb{R}^n} \min\{|x-z|^{1-s},1\} \, |\;\! (-\Delta)^s\phi_r(x-z)| \, \dd z \le 
C(1+r^{1-3s}).
\]
Finally, we obtain 
\[
\big| (-\Delta)^s(u(t,x)-u(s,x)) \big| \le C \bigg[ r^{1-s} + \frac{(t-s)^{1+\alpha}}{r^{2s}} + C (t-s) (1+r^{1-3s}) \bigg].
\] Also,  since $\alpha<(1-s)/2s$, we have
\[
\alpha \le \frac{(1-s)(1+\alpha)}{1+s} \le 1 + \frac{(1-3s)(1+\alpha)}{1+s}
\] Therefore, the choice $r\coloneqq (t-s)^{(1+\alpha)/(1+s)}$ ensures
\[
|(-\Delta)^su(t,x)-(-\Delta)^su(s,x)|\leq C(t-s)^{(1+\alpha)\frac{1-s}{1+s}}.
\]
Analogously, we estimate
\[
\begin{split}
 |\pazocal{R}u(t,x)-\pazocal{R}u(s,x)|\leq&\left| \int_{\mathbb{R}^n} \big[\pazocal{R}u(t,x)- \pazocal{R}u(t,z)\big]\phi_r(x-z) \, \dd z \right| \\ 
   &+ \left|\int_{\mathbb{R}^n} \big[\pazocal{R}u(t,z)- \pazocal{R}u(s,z) \big] \phi_r(x-z) \, \dd z \, \right| \\
   &+ \left| \int_{\mathbb{R}^n} \big[\pazocal{R}u(s,x)- \pazocal{R}u(s,z)\big] \phi_r(x-z) \, \dd z. \right|,
\end{split}
\]
Once again, the first and third integrals are bounded by $C r^{1-s+\gamma}$. For the second integral, we split the integral into
\begin{equation}\label{splitR}
\begin{split}
C\Bigg(\left|\int_{\mathbb{R}^n} \big[\pazocal{I}u(t,z)- \pazocal{I}u(s,z) \big] \phi_r(x-z) \, \dd z \, \right|&+\left|\int_{\mathbb{R}^n} \big[\nabla u(t,z)- \nabla u(s,z) \big] \phi_r(x-z) \, \dd z \, \right|\\
&+\left|\int_{\mathbb{R}^n} \big[ u(t,z)- u(s,z) \big] \phi_r(x-z) \, \dd z \, \right|\Bigg).
\end{split}
\end{equation}
Notice that the third term in \eqref{splitR} is Lipschitz-in-time, thus there is nothing to prove. For the remaining terms, we proceed analogously and bound them by
\[
\begin{split}
	C(t-s)\Bigg(1+\frac{(t-s)^{\alpha}}{r}+\frac{(t-s)^{\alpha}}{r^{2\sigma}}&+\frac{1}{r^{n+1}}\int_{B_r}|w|^{1-s} \, \dd w + \int_{B_1\setminus B_r}|w|^{-s-n} \, \dd w\\
&+ \frac{1}{r^{n+2\sigma}}\int_{B_r}|w|^{1-s} \, \dd w + \int_{B_1\setminus B_r}|w|^{1-2\sigma-s-n} \, \dd w\Bigg).
\end{split}
\] 
Now, we estimate the integrals above:
\[\begin{split}
\frac{1}{r^{n+1}}\int_{B_r}|w|^{1-s} \, \dd w + \int_{B_1\setminus B_r}|w|^{-s-n} \, \dd w &\leq C\left(1+\frac{1}{r^s}\right),\\
\frac{1}{r^{n+2\sigma}}\int_{B_r}|w|^{1-s} \, \dd w + \int_{B_1\setminus B_r}|w|^{1-2\sigma-s-n} \, \dd w \leq& C\Bigg(\frac{1}{r^{2\sigma+s-1}}
\\&+\begin{cases}1 & \text{ if } \sigma<(1-s)/2;\\
1+|\log(r)| & \text{ if } \sigma=(1-s)/2;\\
1+r^{1-2\sigma-s} & \text{ if } \sigma>(1-s)/2.
\end{cases}\Bigg)
\end{split}\]
	
Hence, we conclude that, for $\sigma\geq 0$,
\[\begin{split}
	|\pazocal{R}u(t,x)-\pazocal{R}u(s,x)|\leq C r^{1-s+\gamma} +C(t-s)\Bigg(\frac{(t-s)^{\alpha}}{r}+\frac{(t-s)^{\alpha}}{r^{2\sigma}}+\frac{(1}{r^s}+\frac{1}{r^{2\sigma+s-1}}\\
	+\begin{cases}1 & \text{ if } \sigma<(1-s)/2;\\
1+|\log(r)| & \text{ if } \sigma=(1-s)/2;\\
1+r^{1-2\sigma-s} & \text{ if } \sigma>(1-s)/2.
\end{cases}\Bigg)
	\end{split}\]
If $0\leq\sigma<(1-s)/2$, then (recall that $\gamma=s-1/2$)
\[
|\pazocal{R}u(t,x)-\pazocal{R}u(s,x)|\leq C\left(r^{1/2}+(t-s)+\frac{(t-s)^{1+\alpha}}{r}+\frac{(t-s)}{r^s}\right).
\]
Then, by choosing $r\coloneqq (t-s)^{2(1+\alpha)/3}$, we have 
\[
|\pazocal{R}u(t,x)-\pazocal{R}u(s,x)|\leq C\left((t-s)^{\frac{1}{3}(1+\alpha)}+(t-s)^{1-\frac{2s}{3}(1+\alpha)}\right).
\]
We claim that $\frac{1}{3}(1+\alpha)\leq 1-\frac{2s}{3}(1+\alpha)$. Indeed, the claim holds if, and only if $(1+\alpha)(1+2s)\leq 3$. Now, since $\alpha<(1-s)/2s$, we have $(1+\alpha)(1+2s)\leq \frac{3}{2}+s+\frac{1}{2s}$. Moreover, $s+\frac{1}{2s}\leq \frac{3}{2} \iff 2s^2+1-3s<0$, and the latter holds since $1/2<s<1$. Hence, we conclude
\[
|\pazocal{R}u(t,x)-\pazocal{R}u(s,x)|\leq C(t-s)^{(1+\alpha)/3}.
\]
If $(1-s)/2\leq\sigma\leq 1/2$, then
\[
|\pazocal{R}u(t,x)-\pazocal{R}u(s,x)|\leq C\left(r^{1/2}+(t-s)\left(1+|\log(r)|\right)+\frac{(t-s)^{1+\alpha}}{r}+\frac{(t-s)}{r^s}\right).
\]
Once again, by choosing $r\coloneqq (t-s)^{2(1+\alpha)/3}$, we obtain
\[
|\pazocal{R}u(t,x)-\pazocal{R}u(s,x)|\leq C(t-s)^{(1+\alpha)/3}.
\]
Finally, if $\sigma>1/2$, then
\[
|\pazocal{R}u(t,x)-\pazocal{R}u(s,x)|\leq C\left(r^{1-\sigma}+(t-s)+\frac{(t-s)^{1+\alpha}}{r^{2\sigma}}+\frac{(t-s)}{r^{2\sigma+s-1}}\right).
\]
Choosing $r\coloneqq (t-s)^{(1+\alpha)/(1+\sigma)}$, we obtain
\[
|\pazocal{R}u(t,x)-\pazocal{R}u(s,x)|\leq C\left((t-s)^{(1+\alpha)\frac{1-\sigma}{1+\sigma}}+(t-s)^{1-(1+\alpha)\frac{2\sigma+s-1}{1+\sigma}}\right).
\]
We claim that $(1+\alpha)\frac{1-\sigma}{1+\sigma}\leq 1-(1+\alpha)\frac{2\sigma+s-1}{1+\sigma}$. Indeed, the claim holds if, and only if $(1+\alpha)(\sigma+s)\leq 1+\sigma$. Once again, since $\alpha<(1-s)/2s$, we have $(1+\alpha)(\sigma+s)\leq \frac{1+\sigma}{2}+\frac{s}{2}+\frac{\sigma}{2s}$. Moreover, $\frac{s}{2}+\frac{\sigma}{2s}\leq  \frac{1+\sigma}{2} \iff s^2-(1+\sigma)s+\sigma\leq 0$, and the latter holds since $1/2<\sigma<s$. Hence, we conclude 
\[
|\pazocal{R}u(t,x)-\pazocal{R}u(s,x)|\leq C(t-s)^{(1+\alpha)\frac{1-\sigma}{1+\sigma}}.
\]
Thus, by \eqref{exp}, we conclude the lemma.
\end{proof}

We are now ready to prove our main regularity result.

\begin{proof}[Proof of \Cref{maintheorem}]  The global Lipschitz regularity follows from \Cref{cor:lipschitz}. Given $\alpha \in \big( 0, \tfrac{1-s}{2s} \big)$, denote by $\Phi$ the affine function 
	\[
	\Phi(\alpha)\coloneqq (1+\alpha)\frac{1-s}{1+s}
	\] which is strictly increasing and satisfies $\Phi(\frac{1-s}{2s})=\frac{1-s}{2s}$. By \eqref{firstregintime}, we can apply \Cref{improvedreg}, which gives
\[
\big[(-\Delta)^su-\pazocal{R}u\big] \chi_{\{u=\psi\}} \in C^{\Phi\left(\frac{1-s}{1+s}-0^+\right),1-s}_{t,x}((0,T]\times\mathbb{R}^n).
\] Then, by \eqref{regalphabeta} and interpolation inequalities, we have
\[\partial_t u, \ (-\Delta)^su \ \in \ C^{\Phi\left(\frac{1-s}{1+s}-0^+\right),1-s}_{t,x}((0,T]\times\mathbb{R}^n),
\] since $\Phi(\alpha)<\frac{1-s}{2s}$ for $\alpha<\frac{1-s}{2s}$. Next, we apply \Cref{improvedreg} and \eqref{regalphabeta} iteratively  to obtain
\[\partial_t u, \ (-\Delta)^su \ \in \ C^{\Phi^n\left(\frac{1-s}{1+s}-0^+\right),1-s}_{t,x}((0,T]\times\mathbb{R}^n),
\]
which combined with \cite[Estimate A.5]{CaFi-08} gives
\[
(-\Delta)^su \in C^{\frac{1-s}{2s},1-s}_{t,x}((0,T]\times\mathbb{R}^n).
\]  Since $\Phi^{n} \left(\frac{1-s}{1+s}-0^+\right) \longrightarrow \frac{1-s}{2s}$ as $n\rightarrow \infty$, we also conclude
\[
\partial_t u \in C^{\frac{1-s}{2s}-0^+,1-s}_{t,x}((0,T]\times\mathbb{R}^n). \qedhere
\]
\end{proof}

\appendix
\section{Regularity results for \texorpdfstring{$\partial_t+(-\Delta)^s = f$}{} with \texorpdfstring{$f\in L^\infty$}{}}
We now adress the approximation of a solution of \eqref{prin} by a solution of \eqref{aprox}. The main ideas are found in \cite{ellipticcase}. We already used the regularities \eqref{regularityheateq} and \eqref{regalphabeta} of the fractional heat equation. However, in both cases the source $f$ is Hölder continuous. Nonetheless, we need a regularity result when $f$ is merely a bounded function in spacetime. Namely, for $\partial_t v+(-\Delta)^sv=f$, we have
\begin{equation}\label{fbounded}
\|v\|_{C^{1-0^+}((0,T];L^\infty(\mathbb{R}^n))}+\|v\|_{L^\infty((0,T];C^{2s-0^+}(\mathbb{R}^n))}\leq C(1+\|f\|_{L^\infty((0,T]\times\mathbb{R}^n)}).
\end{equation}
In order to show \eqref{fbounded}, we proceed as in \cite{CaFi-08}: we notice that 
\[
v(t,x)=\Gamma_s(t)\ast v(0)+\int_0^t \Gamma_s(t-\tau)\ast f(\tau)\,\dd \tau,
\]
where $\Gamma_s(t,y)$ is fundamental solution of the fractional heat equation and it behaves as
\begin{equation}\label{Gammabehavior}
\begin{split}
|\Gamma_s(t,y)|\sim \frac{t}{t^{\frac{n+2s}{2s}}+|y|^{n+2s}},& \quad |\partial_t\Gamma_s(t,y)|\leq C\frac{1}{t^{\frac{n+2s}{2s}}+|y|^{n+2s}},\\
|\nabla_y\Gamma_s(t,y)|\leq C\frac{1}{|y|}\frac{t}{t^{\frac{n+2s}{2s}}+|y|^{n+2s}},& \quad |D^2_y\Gamma_s(t,y)|\leq C\frac{1}{|y|^2}\frac{t}{t^{\frac{n+2s}{2s}}+|y|^{n+2s}}.
\end{split}
\end{equation}
Since the initial condition of \eqref{prin} is well-behaved (namely, it satisfies \eqref{psi}), the first term is smooth. Thus, we only need to estimate the source term. In order to do it, we will need the following estimates:
\begin{itemize}
\item  there exists a constant $C>0$ such that, for all $h> 0$,
\begin{equation}\label{a10}
\int_{\mathbb{R}^n}\frac{h}{h^{\frac{n+2s}{2s}}+|z|^{n+2s}}\,\dd z\leq C(1+h);
\end{equation}

\item there exists a constant $C>0$ such that, for all $h> 0$,
\begin{equation}\label{a11}
\int_0^t\frac{t-\tau}{(t-\tau)^{\frac{n+2s}{2s}}+h^{n+2s}}\,\dd \tau\leq 
C\min\{h^{-n-2s},h^{-n+2s}\}.
\end{equation}
\end{itemize} 
The proof of both are very simple: for \eqref{a10}, one splits the integral into $B_{h^{1/2s}}$ and $\mathbb{R}^n\setminus B_{h^{1/2s}}$, thus
\[
\int_{\mathbb{R}^n}\frac{h}{h^{\frac{n+2s}{2s}}+|z|^{n+2s}}\,\dd z\leq \frac{1}{h^{n/2s}}|B_{h^{1/2s}}|+h\int_{\mathbb{R}^n\setminus B_{h^{1/2s}}}\frac{1}{|z|^{n+2s}}\,\dd z\leq C(1+h).
\]
For \eqref{a11}, if $h\geq 1$, the bound is trivial, since the integrand is bounded by $h^{-n-2s}$; otherwise, if $h\in(0,1]$, we split the integral into $[0,t-h^{2s}]$ and $[t-h^{2s},t]$, thus for $n\geq 2$ we have
\[
\int_0^t\frac{t-\tau}{(t-\tau)^{\frac{n+2s}{2s}}+h^{n+2s}}\,\dd \tau\leq \int_0^{t-h^{2s}}(t-\tau)^{-n/2s} \,\dd \tau+\frac{1}{h^{n+2s}}\int_{t-h^{2s}}^t(t-\tau) \,\dd \tau\leq C h^{-n+2s}.
\]

For the time regularity of \eqref{fbounded}, notice that for $u<t$, we have
\[
|v(t)-v(u)|\leq C\left(\int_u^t |\Gamma_s(t-\tau)\ast f(\tau)|\,\dd \tau+\int_0^u |(\Gamma_s(t-\tau)-\Gamma_s(u-\tau))\ast f(\tau)|\,\dd \tau\right).
\]
By \eqref{Gammabehavior} and \eqref{a10} with $h=t-\tau$, the first term can be bounded by
\[
C\|f\|_{L^\infty((0,T]\times\mathbb{R}^n)}\int_u^t [1+(t-\tau)]\,\dd \tau\leq C\|f\|_{L^\infty((0,T]\times\mathbb{R}^n)}(t-u),
\]
while the second term can be bounded by
\[\begin{split}
&C\|f\|_{L^\infty((0,T]\times\mathbb{R}^n)}\int_0^u\min\{t-u,u-\tau\}(1+(u-\tau)^{-1})\,\dd \tau\\
&\leq C\|f\|_{L^\infty((0,T]\times\mathbb{R}^n)}\left((t-u)\int_0^{u-(t-u)}(1+(u-\tau)^{-1})\,\dd \tau+\int_{u-(t-u)}^u(1+(u-\tau))\,\dd \tau\right)\\
&\leq C\|f\|_{L^\infty((0,T]\times\mathbb{R}^n)}(t-u)|\log(t-u)|.
\end{split}\]

For the space regularity of \eqref{fbounded}, since $2s>1$, we evaluate
\[
|\nabla v(x)-\nabla v(z)|\leq C\|f\|_{L^\infty((0,T]\times\mathbb{R}^n)}\int_0^t\int_{\mathbb{R}^n}|\nabla \Gamma_s(t-\tau,x-y)-\nabla \Gamma_s(t-\tau,z-y)|\,\dd y\,\dd \tau.
\]
We split the space integral into $\{|x-z|\leq |x-y|/2\}$ and $\{|x-z|\geq |x-y|/2\}$. For the first region, by \eqref{Gammabehavior} and \eqref{a11} with $h=|x-y|$, we can bound the integral ($n\geq 2$) by
\[\begin{split}
&C|x-z|\int_{\{|x-z|\leq |x-y|/2\}}|x-y|^{-2}\min\{|x-y|^{-n-2s},|x-y|^{-n+2s}\} \,\dd y\\
&\leq C|x-z|^{1+2s-2}\int_{\{|x-z|\leq |x-y|/2\leq 1\}}|x-y|^{-n} \,\dd y+C|x-z|\int_{\{|x-y|/2\geq 1\}}|x-y|^{-2-n-2s}\,\dd y\\
&\leq C|x-z|^{2s-1}\left|\log|x-z|\right|,
\end{split}\]
while 
For the second region, by \eqref{Gammabehavior}, \eqref{a11} with $h=|x-y|$ and noticing that $\{|x-z|\geq |x-y|/2\}\subset B_{3|x-z|}(x)\cap B_{3|x-z|}(z)$, we can bound the integral by
\[
\int_{B_{3|x-z|}(x)}\frac{1}{|x-y|}|x-y|^{-n+2s}\,\dd y\leq |x-z|^{2s-1}.
\]

\end{document}